\documentclass[a4paper,11pt]{article}

\usepackage{mathrsfs}
\usepackage{amsmath,amscd}

\usepackage{graphicx}
\usepackage{color}
\usepackage{CJK}
\usepackage{amsmath,amsthm,amssymb,amsfonts}
\usepackage[numbers,sort&compress]{natbib}

\usepackage{mathtools}
\mathtoolsset{showonlyrefs=true}

\usepackage[
            a4paper,colorlinks,breaklinks,unicode]{hyperref}

\allowdisplaybreaks

\newcommand{\R}{\mathbb R}

\newtheorem{definition}{Definition}[section]
\newtheorem{theorem}{Theorem}[section]

\newtheorem{lemma}{Lemma}[section]
\newtheorem{remark}{Remark}[section]

\newtheorem{claim}{Claim}[section]
\numberwithin{equation}{section}

\author{Bin Li\textsuperscript{a}\quad   Tian Xiang\textsuperscript{b}\thanks{E-mail:  \href{mailto:txiang@ruc.edu.cn}{\tt txiang@ruc.edu.cn}}\\
\emph{\small \textsuperscript{a}School of Statistics and Data Science, Ningbo University of Technology, Ningbo 315211, PR China;}\\
\emph{\small \textsuperscript{b}School of Mathematics and Institute for  Mathematical Sciences,}\\
\emph{\small Renmin University of China, Beijing 100872, PR China.}\\
  }

\title{Roles of logistic source in a singular  chemotaxis urban crime model}

\date{}

\begin{document}
\maketitle
\begin{minipage}{6.2in}\vskip 1mm
 {\small \textbf{Abstract:}
 Recently, Heihoff introduced logistic source into the Short \textit{et al}  urban crime model  to account for  fierce competition among criminals. The resulting model reads as
 \begin{equation*}
\left\{
\begin{split}
&u_t= \Delta u-\chi\nabla\cdot\left(u\nabla\ln v\right)- uv+ru-\mu u^2+ h_1,\\
&v_t=\Delta v- v+ uv+h_2,
\end{split}
\right.
\end{equation*}
where  the parameters $\chi,\mu>0$  and $r\in\R$, and  the given functions $h_1,h_2\geq0$. In this paper, we systematically exhibit the roles of logistic source on global boundedness  and convergence of classical solutions, and global existence and eventual smoothness as well as convergence of generalized solutions for this crime model. More specifically,
we first show  that for any $\chi>0$ there exists an explicit $\mu_0(\chi, n)>0$ such that whenever $\mu>\mu_0(\chi,n)$,   the corresponding
 no-flux initial-boundary value problem posed on an $n$-dimensional bounded and smooth domain possesses a unique global and bounded classical solution, and then,  we show that these bounded solutions converge  uniformly  to  its semi-trivial  equilibria under some additional conditions. Furthermore, for $n\leq3$, we show that  for any $\mu>0$  such  problem admits a generalized solution, which becomes smooth and converges eventually. Roughly speaking, this work shows that suitably strong logistic damping  ensures boundedness and convergence of classical solutions, and  weak logistic damping ensures global existence of generalized solution and its eventual smoothness as well as convergence  in two and three  dimensions. The former in particular implies that strong logistic damping can prevent the formation of crime hotspots (spatial patterns). Previous known results in this respect are thus   significantly improved and extended.
}

\vspace{1.5mm}


\vspace{1.5mm}


\end{minipage}

 \section{Introduction and main results}

This paper  investigates the Neumann initial-boundary value problem (IBVP) for  an adaptation  of  logarithmic chemotaxis-type  system  arising
in crime modelling.  In its original version, as proposed by Short  \textit{et al} in \cite{Short2010Pnas,Short2008} to approximate the spatial-temporal urban criminal activity, this
model takes the form
\begin{equation}\label{eq-r0}
\left\{
\begin{split}
&u_t= \Delta u-\chi\nabla\cdot\left(u\nabla\ln v\right)-uv+h_1,\\
&v_t=\Delta v-v+uv+h_2,
\end{split}
\right.
\end{equation}
with a typical choice of $\chi=2$. In this model, $u=u(x,t)$ denotes the density of criminals, $v=v(x,t)$ characterizes the attractiveness  of criminal activities,  $h_1$ describes  the density of  influx of additional criminal agents and $h_2$ denotes external sources of attractiveness.
For more details, we refer the interested readers to \cite{Rodriguez2010,Pan2018} for   Short \textit{et al}  model, and \cite{Pitcher2010,Zipkin2014,Rodriguez2013,Gu2017,Jones2010M3as} for  related models, as well as \cite{D'Orsogna2015,Bellomo2022M3as} for  reviews.

The crime model \eqref{eq-r0} is nonlinearly coupled in the second order and the zeroth order. Since its appearance, it  has attracted much attention of
scholars, because it not only  provides the
information on crime  hotspots, such as   the emergence (\cite{Cantrell2012,Gu2017,Short2010Pnas,Short2010b,Tse2016Ejam}) and the large-amplitude peaks (\cite{Berestycki2014Siam,Berestycki2010,David2013physD,MWM3as2020}) of crime
hotspots,
 but also, it features a nonlinear production term $+uv$, which,  in comparison to the linear production term $+u$ in
 the logarithmic Keller-Segel system (\cite{Keller1970,Keller1971}), remains poorly understood, resulting in limited progress in the mathematical analysis of \eqref{eq-r0}.   Thanks to the latter,
    there are only few  results on the well-posedness: The classical solution  is global provided that either $n=1$ (\cite{WangQ2020JDE,R-Winkler2019Ejam}), or
 $n\geq2$ and $\chi<\frac2n$ (\cite{Freitag2018,ShenJ2019m2as}), or the initial data and the given functions $h_1$ and $h_2$  are assumed to be   small (\cite{TW2021CMS,Ahn2021JDE}). The restriction on $\chi$ can be (slightly) relaxed in the  radial setting  (\cite{Winkler2019Poin,Jiang2022AAm}), and in the generalized framework (\cite{LIXIE2022}). Here, we also  remark that  \cite{LIXIE2022} derived only the global existence of the  generalized  super-solution of  \eqref{eq-r0} in the two-dimensional setting,  and the  global existence of generalized  solution has been  established only for certain modified versions, see \cite{TW2025SCM} in which  $+uv$ in \eqref{eq-r0} is replaced by $+h(u,v)uv$ with $h(u,v)\rightarrow0$ as $\min\{u,v\}\rightarrow\infty$, and see \cite{Heihoff2020zamp} in which  logistic source is incorporated into \eqref{eq-r0}.
 There also exist various results on the qualitative properties of crime
hotspots (e.g., \cite{B-Wei2020Ejam,R-Wang2021,Tse2018siam}) and on the well-posedness of the other variants of  Short \textit{et al} model (e.g., \cite{Manasevich2013,Rodriguez2013,LX2024CVPDE,LX2022Dcdsb,LWX2023MBE}).

 We observe that the crime  model \eqref{eq-r0} can effectively simulate the formation of crime hotspots (cf. \cite{Short2010Pnas,Berestycki2014Siam}), which is supported by the blow-up result in certain sense  with $h_1=0=h_2$ \cite{Fuest2024ccm}. However, from a practical perspective,
  suppressing such formations become imperative because  they pose significant threats to both personal safety and property security.
Therefore,   various mechanisms have been introduced into system \eqref{eq-r0} to prevent the formation of crime
hotspots, such as the law enforcement actions (e.g., \cite{Jones2010M3as,Zipkin2014}), the nonlinear diffusion (\cite{R-Winkler2020M3as,YangX2022NA})
 as well as the generalized logistic source (\cite{Heihoff2020zamp,WangD2022M2as}). In particular, when the generalized  logistic source is incorporated into system \eqref{eq-r0}, the resulting system (\cite{Heihoff2020zamp}) becomes
 \begin{equation}\label{eq-r0C}
\left\{
\begin{split}
&u_t= \Delta u-\chi\nabla\cdot\left(u\nabla\ln v\right)-uv+ru-\mu u^{\beta}+h_1,\\
&v_t=\Delta v-v+ uv+h_2,
\end{split}
\right.
\end{equation}
 where $\chi>0$, $r\in\R$, $\beta\geq2$ and $\mu>0$. If $\beta>2$, then the corresponding  IBVP admits a global and bounded classical solution in the two-dimensional setting (\cite{Heihoff2020zamp}); while the restriction on $\beta$ needs  be strengthened to that $\beta>\max\{2,\frac n4+1\}$ in $n$-dimensional settings with $n\geq3$ (\cite{WangD2022M2as}). As for the typical case $\beta=2$, the IBVP for  \eqref{eq-r0C} possesses a generalized solution in two-dimensional setting (\cite{Heihoff2020zamp}), which becomes smooth and bounded at least eventually (\cite{QLI2023ERA}). In contrast, when  $h_1=0=h_2$, $-uv$ is absent in the first eqution and  $+uv$ in the second equation is replaced by $+u$, the resulting system  reduces to the the logarithmic
Keller-Segel system:
  \begin{equation}\label{eq-ck-log}
\left\{
\begin{split}
&u_t= \Delta u-\chi\nabla\cdot\left(u\nabla\ln v\right)+ru-\mu u^{\beta},\\
&v_t=\Delta v-v+ u.
\end{split}
\right.
\end{equation}
 The corresponding IBVP for \eqref{eq-ck-log} admits a global classical solution, see \cite{Aida2005,Zhao2017} for  $\beta=2$ in the two-dimensional setting and  \cite{DWZ2019NA} for $\beta>2$ in $n$-dimensional settings  with $n\geq3$. It is also shown in   \cite{DWZ2019NA} that the IBVP for \eqref{eq-ck-log} admits global weak solution for $\beta=2$ and $n\geq 3$. It is yet not known whether this weak solution will become smooth or not?

To illustrate our motivation, we also recall some results on the standard chemotaxis-logistic system without singular sensitivity  obtained when changing $\ln v$ to $v$ and specifying $\beta=2$ in \eqref{eq-ck-log}:
 \begin{equation}\label{eq-r0L}
\left\{
\begin{split}
&u_t= \Delta u-\chi\nabla\cdot\left(u\nabla v\right) +ru-\mu u^2,\\
&v_t=\Delta v-v+  u.
\end{split}
\right.
\end{equation}
It is well-known that if either  $\mu>0$ for $n=2$, or $\mu>\widehat{\mu}$ with some $\widehat{\mu}>0$ for $n\geq3$ (\cite{Winkler2010CPDE}), then
 the standard logistic damping is enough to suppress the singularities of solution that observed in \cite{HorWa2001Ejam,HerVel1997Asnspcs,SS2001maa,Winkler2010,Winkler2013JMPA}
   for the minimal Keller-Segel system (\cite{Nanjundiah1973}) corresponding to $r=\mu=0$; further
qualitative findings on logistic damping versus chemotatic aggregation
 can be found in \cite{JinX2018Crmasp,Xiang2018jmaa,Xiang2018Siam,Xiang2015jde} for  the refinements with respect to parameter setting, in \cite{Cao2017Dcds,Winkler2014Jde} for the global asymptotic stability of the spatial homogenous equilibria.    We also remark that weak logitsic damping ensures  the global existence of weak solutions for $n\geq 3$, and also  warrants its eventual smoothness provided that $r>0$ is sufficiently small and $n=3$ (\cite{Lankeit2015jde}).

Motivated by the results on \eqref{eq-r0C},  \eqref{eq-ck-log} and \eqref{eq-r0L}, it is quite natural to ask  whether standard logistic damping can prevent the formation of crime hotspots for the chemotaxis  urban crime system?
Here, we shall provide some satisfactory  answers. Beyond that, we shall study the effect of logistic damping  versus  singular chemotatic aggregation on the dynamics of classical solutions as well as generalized weak solutions to  the
IBVP for \eqref{eq-r0C} with $\beta=2$:
\begin{equation}\label{eq-r01}
\left\{
\begin{split}
&u_t= \Delta u-\chi\nabla\cdot\left(u\nabla\ln v\right)- uv+ru-\mu u^2+h_1,&\ x\in \Omega,\ t>0, \\
&v_t=\Delta v-v+uv+h_2, &\ x\in \Omega,\ t>0,\\
&\nabla u\cdot\nu=\nabla v\cdot\nu=0,&\ x\in\partial\Omega,\ t>0,\\
&u(x,0)=u_0(x),\ v(x,0)=v_0(x),&\ x\in\Omega.
\end{split}
\right.
\end{equation}
Here and below, $\Omega\subset\R^n$ ($n\ge2$) is a bounded domain  with smooth boundary $\partial\Omega$,  $\nu$ denotes the unit exterior normal vector to the boundary $\partial\Omega$ and the initial data $(u_0,v_0)$ fulfills that
\begin{equation}\label{eq-iva}
u_0\in \mathcal{C}^0(\overline{\Omega}) \ \ \text{ and  }\  u_0>0 \ \text{ on } \ \overline\Omega,\ \ \ v_0\in W^{1,\infty}(\overline{\Omega}) \ \text{ and  }\  u_0>0 \ \text{ on } \ \overline\Omega.
\end{equation}
As for the given source functions, we conveniently assume throughout this paper that
\begin{align}\label{eq-iva0}
0\leq h_i\in \mathcal{C}^1(\overline{\Omega}\times[0,\infty))\cap L^\infty(\Omega\times([0,\infty)),\quad i=1,2.
\end{align}

In such setup, we   state our main findings on the global dynamics of \eqref{eq-r01}.  We first state that suitably strong logistic damping can  prevent blow-up of classical solutions in any dimensions.

\begin{theorem}\label{th-global}
Let $\chi>0$, $r\in\R$ and  $n\geq2$, and let \eqref{eq-iva}-\eqref{eq-iva0} hold. There exists an explicit $\mu_0=\mu_0(\chi, n)>0$ (See (R2) or \eqref{mu0-def} for its formula) such that
whenever $\mu>\mu_0$, then  the IBVP \eqref{eq-r01} possesses a unique classical solution $(u,v)$ on $\overline{\Omega}\times[0,\infty)$, with the properties that  $u(x,t)> 0$ and $v(x,t)>0$ for any $x\in\overline{\Omega}$ and any $t>0$, and for any $T>0$,
\begin{align}\label{eq-esm}
&u\in \mathcal{C}^0(\overline{\Omega}\times[0,T])\cap \mathcal{C}^{2,1}(\overline{\Omega}\times(0,T]),\quad v\in \mathcal{C}^0(\overline{\Omega}\times[0,T])\cap \mathcal{C}^{2,1}(\overline{\Omega}\times(0,T]).
\end{align}
Moreover, under the additional assumption that
  \begin{align}
 & \inf_{t>0}\int_\Omega h_2(x,t)dx>0,\label{eq-iva1}
 \end{align}
there exists $C=C(u_0, v_0, \chi,r,\mu, h_1,h_2)>0$ independent of $t$ such that
\begin{align}\label{eq-bedds}
\|u(\cdot,t)\|_{L^\infty}+\|v^{-1}(\cdot,t)\|_{L^\infty}+\|v(\cdot,t)\|_{W^{1,\infty}}\leq C,\quad t>0.
\end{align}
\end{theorem}

\begin{remark}[Notes on strong  logistic damping  prevents blow-up of classical solutions for \eqref{eq-r01}]
\
\
\
\begin{itemize}
\item [(R1)]For $n=1$,  the global existence and boundnesness of classical solutions for \eqref{eq-r01} with $r=\mu=0$ have been studied in \cite{WangQ2020JDE,R-Winkler2019Ejam}.  This is the reason why we assume $n\geq 2$.
\item[(R2)]To illustrate the explicit role of logistic damping on global existence and boundedness of classical solutions, we compute two rough bounds of $\mu_0$ for $n\in \{2, 3\}$ as follows: if
\[
\mu>\tilde\mu_0(\chi,n):=
\begin{cases}
\begin{split}
 20666.6+1.1232\left[(\chi-2)^2+12\right]^\frac32,  &\  \text{if }   n=2, \\
26477.2+1.2757\left[(\chi-2)^2+12\right]^\frac32,  &\  \text{if }   n=3,
\end{split}
\end{cases}
\]
then  the IBVP  \eqref{eq-r01} has a unique global classical solution, which  is bounded provided that \eqref{eq-iva1} is satisfied. Coincidentally, such bound $\tilde\mu_0(\chi,n)$ (see Remark \ref{bdd-rem} for its computations) achieves its global minimum at $\chi=2$, which appears in the original model (\cite{Short2010Pnas,Short2008}).  In a future work, it is quite interesting to see whether there exists or not a bound for $\mu_0(\chi,n)$ with the property that $\mu_0(0,n)=0$?

\item[(R3)] Our findings significantly improve  \cite{Heihoff2020zamp, WangD2022M2as}  on super-quadratic damping preventing blow-up for \eqref{eq-r0C}, and improve  \cite{DWZ2019NA}  on super-quadratic degradation preventing blow-up for \eqref{eq-ck-log} with $n\geq 3$ as well, since our arguments surely work for less complex models like \eqref{eq-ck-log}.
\item[(R4)] Although the $v$-component has  \textit{a-priori} positive lower bound (which may lead one to think that  \eqref{eq-r01} is similar to \eqref{eq-r0L}), the second equation in \eqref{eq-r01}  possesses a more complex nonlinear mechanism $+uv$ than that in \eqref{eq-r0L}, which leads to that the strategies used for the latter (e.g., \cite{Winkler2010CPDE,JinX2018Crmasp}) are unavailable and thereby the mathematical analyses of \eqref{eq-r01} are  much more challenging. Here, to compensate  comparison type  argument,  one of our  key steps is  to track the
time evolution of a combined functional, denoted by
\[
\mathcal{E}(t)=\int_{\Omega}u^{p}v^{2p-q}dx+\int_{\Omega} \frac{|\nabla v|^{2p}}{v^q}dx,
\]
with some well-selected $p,q>0$. Through this process, we derive several delicate \textit{a-priori} estimates (see, e.g., Lemmas \ref{le-inequ},  \ref{le-inequ2},  \ref{le-ineqbi}, \ref{le-inev2},  \ref{le-inequi0} and \ref{le-inequ4}) by means of embedding theorems and subtle arguments, which play a crucial role in the subsequent analysis.
\end{itemize}
\end{remark}

Having  boundedness at hand,  we  naturally investigate the longtime behavior of those bounded solutions, which evidently depends on the asymptotic properties of the source  functions $h_1$ and $h_2$. For the nonnegative source function $h_1$, we  assume that
  \begin{equation}\label{eq-iva2}
\int_0^\infty\int_\Omega h_1(x,t)dxdt<\infty;
  \end{equation}
for the nonnegative source function $h_2$, we assume, for some $b>0$,  that
 \begin{equation}\label{eq-iva3}
 \int_0^\infty\int_\Omega|h_2(x,t)-b|^2dxdt<\infty.
 \end{equation}
In such framework,  the IBVP \eqref{eq-r01}  has two possible homogeneous steady states depending on the value of parameters, which  can be classified as follows:
\begin{itemize}
  \item The semi-trivial steady state: $(0,b)$;
  \item The unique positive steady state: $(u_*,v_*)$, which is given by
\begin{align}
u_*=\frac{r+\mu-\sqrt{(\mu-r)^2+4b\mu}}{2\mu}, \  \  \ v_*=\frac{r-\mu+\sqrt{(\mu-r)^2+4b\mu}}{2},\label{eq-sss}
\end{align}
where $r>b$ and $(u_*, v_*)$ is the unique positive solution of the algebraic system:
 \begin{align}\label{eq-ssse}
r-\mu u_*-v_*=0, \ \ \  b-(1-u_*)v_*=0.
\end{align}
\end{itemize}
On the basis of classifications  of semi-trivial   and positive steady states,
our  stability results about these steady states  can be stated as follows.

\begin{theorem}\label{th-long}
Let $(u,v)$ be a bounded solution of \eqref{eq-r01} provided by Theorem \ref{th-global}, and additionally, let $h_1$ and $h_2$ enjoy the spacetime integrability in  \eqref{eq-iva2} and \eqref{eq-iva3}, respectively.
\begin{enumerate}
  \item [(i)] If $r\leq\eta_0$ with $\eta_0$ being the positive lower bound of $v$ given in \eqref{eq-0bddc}, then
\begin{align}\label{eq-smdeu}
\|u(\cdot,t)\|_{L^\infty}+\|v(\cdot,t)-b\|_{W^{1,\infty}}\rightarrow0, \quad \mathrm{as}\,\,\, t\rightarrow\infty;
\end{align}
  \item [(ii)] If $r>b$ with $b>0$ satisfying \eqref{eq-iva3}, define
  \begin{equation}\label{hatmu-def}
      \hat{\mu}_1(\chi)=\inf\left\{\mu>0: \ \ 16b\mu^2-4(r-2b)\chi^2\mu+\left[4r^2-(r-b)\chi^2\right]\chi^2>0\right\},
  \end{equation}
  and $\mu_1=\max\{\hat{\mu}_1(\chi), \ \ \mu_0(\chi,n)\}$ with $\mu_0(\chi,n)$ determined in Theorem \ref{th-global}, then for $\mu>\mu_1$, it holds that
\begin{align}\label{eq-smdeu2}
\|u(\cdot,t)-u_*\|_{L^\infty}+\|v(\cdot,t)-v_*\|_{W^{1,\infty}}\rightarrow0, \quad \mathrm{as}\,\,\, t\rightarrow\infty,
\end{align}
where $(u_*,v_*)$ is given by \eqref{eq-sss}, and  it  solves \eqref{eq-ssse}.
\end{enumerate}
\end{theorem}

\begin{remark}[Notes on the main idea and procedures on convergence of bounded solutions]

\noindent
\begin{itemize}
\item [(R5)]
One of the key steps in the proof of uniform convergence in Theorem \ref{th-long} is the derivation of the $L^2$
  decay estimate for the differences $u-a_u$ and $v-a_v$, where $(a_u,a_v)$ is either $(0,b)$ or $(u_*,v_*)$. In the case of (i),  this follows readily from spacetime integration of the $u$-equation; in the case of (ii),  we shall examine the time evolution of the following functional
\begin{equation*}
\mathcal{E}(t):= \int_{\Omega} \left(u+v-u_*-v_*-u_*\ln \frac{u}{u_*}-v_*\ln \frac{v}{v_*}\right)dx, \quad t>0.
\end{equation*}
Then the condition $\mu>\hat \mu_1(\chi)$ enables us to derive the said $L^2$-decay,  see Lemma \ref{le-dUulnu}.  Thereafter, by combining interpolation inequalities with semigroup estimates, we establish the desired decay estimates, as detailed in Lemmas \ref{le-smdeu0}, \ref{le-usmdeu0} and \ref{le-pusmde}.
Surely, the convergence rate in Theorem   \ref{th-long}  depends on the asymptotic behavior of $h_1$ and $h_2$. It will also be meaningful  to explore the convergence rate.
\end{itemize}
\end{remark}

In Theorems \ref{th-global} and \ref{th-long}, we have illustrated that  strong logistic damping enhances global boundedness and convergence  of classical solutions to the IBVP \eqref{eq-r01}. Next,  we extend to study the  role of weak  logistic damping on the global existence of weak solutions  and their eventual smoothness as well as convergence. In this respect, we shall show that for any $\mu>0$, the IBVP  \eqref{eq-r01} admits a generalized solution, which becomes eventually bounded and smooth, and converges to a steady state when $n\leq3$, as stated more specifically below.

\begin{theorem}\label{th-ggs}
Let $\chi>0$, $r\in\R$ and  $n\geq2$, and let \eqref{eq-iva}-\eqref{eq-iva0} hold. Then for any $\mu>0$, the  IBVP \eqref{eq-r01} admits at least one global generalized solution $(u,v)$
in the sense of Definitions \ref{def-gs} below. Moreover, if   \eqref{eq-iva1} and \eqref{eq-iva2} hold, $n\leq3$ and  $r<\widetilde{\eta}$ with $\widetilde{\eta}>0$ (being a uniform lower bound for the $v$-component) given in \eqref{eq-0bddca},  then the following statements hold:
\begin{itemize}
  \item (eventual smoothness) There exists $T_0>0$   such that
\begin{align}\label{eq-esmes}
&u\in  \mathcal{C}^{2,1}(\overline{\Omega}\times[T_0,\infty)),\quad v\in  \mathcal{C}^{2,1}(\overline{\Omega}\times[T_0,\infty)),
\end{align}
which $(u,v)$ solves the boundary
value problem in \eqref{eq-r01} classically on $\Omega\times[T_0,\infty)$;
  \item (convergence) If   \eqref{eq-iva3} is valid, then the convergence property
\eqref{eq-smdeu} also holds.
\end{itemize}
\end{theorem}

\begin{remark}[Notes on weak logistic damping  on generalized solutions for \eqref{eq-r01}]
\
\
\
\begin{itemize}
\item[(R6)] For $n=2$, the generalized solvability and   eventual smoothness were established respectively in \cite{Heihoff2020zamp} and \cite{QLI2023ERA}, where $u\in L^2(\Omega\times(0,\infty))$  ensuring  $v\in L^q(\Omega\times (0, \infty))$ for any $q<\infty$  played a crucial role. However, this is unavailable for  $n\geq3$. Here,  to confront the emerging difficulty, our novelty consists  of  establishing a
series of convergence and asymptotic properties for the approximate solutions, see, e.g.,  Lemmas \ref{le-couvL1}-\ref{le-Vubed}.
\item[(R7)] Theorem \ref{th-ggs} extends and  improves \cite{DWZ2019NA} for \eqref{eq-ck-log} with $\beta=2$ by showing eventual smoothness and convergence for the derived global weak solution for $n=3$.
\item[(R8)] We underline that, for small  $\mu>0$, the possible  singularities in finite time for (generalized) solutions of
\eqref{eq-r01} can not be ruled out by Theorems \ref{th-global} and  \ref{th-ggs} even for $n\in\{2,3\}$. It is an interesting open problem to  rule out such singularities in future works.
\item[(R9)] To establish the global existence of generalized solutions as stated in Theorem \ref{th-ggs}, our strategy consists of three main steps. The first is to choose an appropriate generalized framework. The second is to derive a series of uniform a priori estimates that ensure the desired convergence (see, e.g., Lemmas \ref{le-coL1} and  \ref{le-couvL1}). The final step is to pass to the limit in the approximate problem, thereby obtaining the existence of generalized solutions (see, e.g., Lemma \ref{le-vw}).
\item[(R10)] The central issue in the proof of the 3D eventual smoothness of generalized solutions consists in finding a uniform waiting time $t_0>0$ for which $\|u_\varepsilon(\cdot,t)\|_{L^2}$ becomes and stays appropriately small for all $t\geq t_0$ and for all $\varepsilon\in (0,1)$. To this end, we proceed in two steps. The first is to find $t_0>0$ such that $\|u_\varepsilon(\cdot,t_0)\|_{L^2}$ is suitably small (Lemma \ref{le-ul2int}). The second is to apply a contradiction argument together with the continuity method to prove that $\|u_\varepsilon(\cdot,t)\|_{L^2}$ remains relatively small for every $t>t_0$ (see Claim \ref{le-claim}); in this step, we shall establish delicate estimates for $\|u_\varepsilon(\cdot,t)\|_{L^2}$ for $t\geq s$ with arbitrary $s>0$ (see Lemma \ref{le-gu2L2}), as well as for $\|\nabla v_\varepsilon(\cdot,t)\|_{L^q}$ with $q\in(3,6)$ for $t\in(t_0,T_\varepsilon)$ with $T_\varepsilon$ being the supremum of time within which $\|u_\varepsilon(\cdot,t)\|_{L^2}$ is under control (see Lemma \ref{le-gvL2}).
\end{itemize}
\end{remark}

The rest of this paper is arranged as follows. Global existence and boundedness of classical solutions are established in  Section \ref{se-pre} and  their long-term behaviors are investigated   in Section \ref{se-ape}.
Section \ref{se-esm} is devoted to showing the global existence and eventual smoothness of generalized solutions.
In the  following texts, for the sake of convenience and without introduction, we shall use $C$ and $C_i$, $i=1,2,...$, to denote positive generic constants depending at most on $u_0,  v_0, \chi, r, n,\mu, \Omega, h_1, h_2$;  in many places, we shall omit integration variables like $dx$, $dt$ etc.,  and shall use frequently abbreviations like, for a function $f$,
	$$
		\|f\|_{L^p}=\|f(\cdot, t)\|_{L^p}=\left(\int_\Omega|f(x,t)|^pdx\right)^\frac{1}{p}.
	$$

\section{Global existence and boundedness of classical solutions}\label{se-pre}

In this section, we shall establish the global existence and boundedness of classical solutions to the  IBVP  \eqref{eq-r01}. We begin with the local well-posedness, which is ensured by the contraction mapping principle and  the well-known properties of the Neumann heat
semigroup as well as parabolic regularity,   as illustrated e.g. in \cite{Winkler2011, Rodriguez2010, Aida2005, Freitag2018} in smiliar settings.

\begin{lemma}\label{le-lolex}
Let the assumptions \eqref{eq-iva}-\eqref{eq-iva0} hold.
There exist a time $T_{\max}\leq\infty$ and a unique pair $(u,v)$ of functions, with the properties that $u>0$ and $v>0$ in $\overline{\Omega}\times[0,T_{\max})$, and that for any   $q>n$
 \begin{equation*}
\left\{
\begin{split}
&u\in \mathcal{C}^0\big(\overline{\Omega}\times[0,T_{\max})\big)\cap \mathcal{C}^{2,1}\big(\overline{\Omega}\times(0,T_{\max})\big),\\
&v\in  \mathcal{C}^0\big(0,T_{\max}; W^{1,q}(\Omega)\big)\cap \mathcal{C}^{2,1}\big(\overline{\Omega}\times(0,T_{\max})\big),
\end{split}
\right.
  \end{equation*}
such that $(u,v)$ solves the IBVP \eqref{eq-r01} classically in $\overline{\Omega}\times(0,T_{\max})$. Moreover, if $T_{\max}<\infty$, then for any $q>n$ it holds that
\begin{align}\label{eq-blowup}
\limsup\limits_{t\rightarrow T_{\max}}\Bigr(\|u(\cdot,t)\|_{L^\infty}+\|v(\cdot,t)\|_{W^{1,q}}+\|v^{-1}(\cdot,t)\|_{L^\infty}\Bigr)=\infty.
\end{align}
\end{lemma}

Thanks to the  extensibility  criterion \eqref{eq-blowup},
 we shall dedicate ourselves to showing there exists $q>n$ such that
\begin{align}\label{eq-bpm}
\limsup\limits_{t\rightarrow T_{\max}}\Bigr(\|u(\cdot,t)\|_{L^\infty}+\|v(\cdot,t)\|_{W^{1,q}}+\|v^{-1}(\cdot,t)\|_{L^\infty}\Bigr)<\infty,
\end{align}
which enables us to infer  $T_{\max}=\infty$.
Consequently, we first focus on  the bound of $\|v^{-1}(\cdot,t)\|_{L^\infty}$, which will play an essential role throughout our analysis. In the sequel, we shall tacitly proceed with the conditions of Lemma \ref{le-lolex} so that we have a unique local classical solution of \eqref{eq-r01}.

\begin{lemma}\label{le-glex0}
For any $T>0$ there exists $C_T>0$ such that for $\widetilde{T}:=\min\{T,T_{\max}\}$ it holds that
\begin{align}\label{eq-0bdd}
v(x, t)\geq  C_T \quad\mathrm{for\,\,all}\,\,\, x\in\Omega\,\,\,\mathrm{and}\,\,\, t\in(0,\widetilde{T}),
\end{align}
with the property that
\begin{align}\label{eq-0bddc}
\eta_0:=\inf_{(x,t)\in \Omega\times (0, \widetilde{T})}v(x,t)>0,\quad  \mathrm{if}\,\,\, \eqref{eq-iva1}\,\,\,\mathrm{is\,\, valid}.
\end{align}
That is, under \eqref{eq-iva1}, $v$ has a uniform positive lower bound on $\bar\Omega\times [0,T_{\max})$.
\end{lemma}

\begin{proof} Using the nonnegativity of $u,v$ and $h_2$, and then applying  the comparison principle and Hopf boundary point lemma to the $v$-equation in \eqref{eq-r01}, we find that
\[
v(x,t)\geq \Bigr(\inf_{x\in\Omega} v_0(x)\Bigr) e^{-t}\geq \Bigr(\inf_{x\in\Omega} v_0(x)\Bigr) e^{-T}, \ \ \  \forall (x,t)\in \Omega \times (0, \tilde{T}),
\]
yielding directly \eqref{eq-0bdd}.

Notice that the the variation-of-constants formula  applied to $v$ yields
\begin{align*}
    v(\cdot,t)=e^{t(\Delta-1)}v_0+\int_0^te^{(t-s)(\Delta -1)}(uv+h_2)ds.
\end{align*}
Then by a slight adaptation of the proof of \cite[Lemma 3.1]{LX2023M3as}, it is  easy to get the uniform lower bound \eqref{eq-0bddc} by employing \eqref{eq-iva1} and the well-known pointwise
positivity property of the Neumann heat semigroup  (e.g., \cite[Lemma 2.3]{XiangT2022Jde} and \cite[Lemma 3.1]{WinklerH2013m3as}). Indeed, there exists a positive and increasing  function $\xi:(0,\infty)\rightarrow (0,\infty)$ such that, for all $(x,t)\in \Omega \times (0, T_{\max})$,
\[v(x,t)\geq \sup_{\sigma\in (0, T_{\max})}\min\left\{\xi(\sigma)\inf_{s>0}\int_\Omega h_2(x,s)dx, \ \  \Bigr(\inf_{x\in\Omega} v_0(x)\Bigr) e^{-\sigma}\right\},
\]
simply giving rise to \eqref{eq-0bddc}.
\end{proof}

Next, we  derive the starting $L^1$-boundedness of solutions.

\begin{lemma}\label{le-uL1}
 There exists $C>0$ independent of $t$ such that
\begin{align}\label{eq-uL1}
\|u(\cdot, t)\|_{L^1}+\|v(\cdot, t)\|_{L^1}\leq C(1+\mu^{-1}),\quad \forall t\in(0,T_{\max}).
\end{align}
\end{lemma}

\begin{proof}
Adding  the first two equations in \eqref{eq-r01} and then integrating over $\Omega$  by parts, one has
\begin{align*}
\begin{split}
 \frac{d}{dt}\int_\Omega (u+v) +  \int_\Omega (u+v)
 &=\int_\Omega \left((r+1)u-\mu u^2\right)+\int_\Omega (h_1+h_2)\\
&\leq \frac{(r+1)_+^2}{2\mu}|\Omega|+\int_\Omega (h_1+h_2).
\end{split}
\end{align*}
Multiplying it by $e^t$  and then integrating, one has
$$
\int_\Omega (u+v)\leq \max\left\{ \frac{(r+1)_+^2}{2\mu}|\Omega|+\|h_1+h_2\|_{L^\infty(\Omega\times(0,\infty)}|\Omega|, \  \  \int_\Omega (u_0+v_0)\right\},
$$
from which \eqref{eq-uL1} follows.
\end{proof}

\subsection{\textit{A-priori} estimates on solutions}
For $\alpha,\beta\in\R$, we now track the time evolution of  the following key mixed-type functional:
 \begin{align}\label{eq-inequ}
\int_{\Omega} u^\alpha (\cdot,t)v^\beta(\cdot,t),   \quad \forall t\in(0,T_{\max}),
\end{align}
which will enable us to lift the $L^1$-bounds of solutions to $L^p$-bounds.
\begin{lemma}\label{le-inequ}
Let $p>1$ and $q\in[p,2p)$. For all $T>0$ there exists $C_T>0$ such that for $\widetilde{T}:=\min\{T,T_{\max}\}$ it holds that
 \begin{align}\label{eq-ineLu}
&\frac{d}{dt}\int_{\Omega}u^{p}v^{2p-q} +(2p-q)\int_{\Omega}u^pv^{2p-q}+\frac{p(p-1)}{2}\int_\Omega u^{p-2}v^{2p-q}|\nabla u|^2\nonumber\\
\leq&\left\{(2p-q)(p\chi-2p+q+1)+\frac{p}{2(p-1)}[\chi(p-1)-2(2p-q)]^2\right\}\int_\Omega u^pv^{2p-q-2}|\nabla v|^2\nonumber\\
&+\left(-\frac p2\mu+2p-q\right)\int_{\Omega}u^{p+1}v^{2p-q}+C_T(1+\mu^{-1}),\quad \forall t\in(0,\widetilde{T}),
\end{align}
 where $\sup_{T>0}C_T<\infty$ if  \eqref{eq-iva1} is valid.
 \end{lemma}

\begin{proof}
Invoking \eqref{eq-r01} and integrating by parts, we  find  that
 \begin{align*}
\frac{d}{dt}\int_{\Omega}u^pv^{2p-q}=&p\int_{\Omega}u^{p-1}v^{2p-q}\left\{\Delta u-\chi\nabla\cdot\left(u\nabla\ln v\right)- uv+ru-\mu u^2+h_1\right\}\\
&+(2p-q)\int_{\Omega}u^pv^{2p-q-1}\left\{\Delta v-v+ uv+h_2\right\}\\
=&-p(p-1)\int_\Omega u^{p-2}v^{2p-q}|\nabla u|^2\\
&+p[\chi(p-1)-2(2p-q)]\int_\Omega u^{p-1}v^{2p-q-1}\nabla u\cdot\nabla v\\
&+[p(2p-q)\chi-(2p-q)(2p-q-1)]\int_\Omega u^pv^{2p-q-2}|\nabla v|^2\\
&-p\int_{\Omega}u^{p-1}v^{2p-q}(uv+\mu u^2)+p\int_{\Omega}u^{p-1}v^{2p-q}(ru+h_1)\\
&-(2p-q)\int_{\Omega}u^pv^{2p-q}+(2p-q)\int_{\Omega}u^pv^{2p-q-1}(uv+h_2).
\end{align*}
Since $2p-q>0$ and $p>1$, the Young inequality allows us to deduce
 \begin{align*}
pr\int_{\Omega}u^pv^{2p-q}
\leq& \frac p4\int_{\Omega}u^pv^{2p-q+1}+C\int_\Omega u^p\\
\leq& \frac p4\int_{\Omega}u^pv^{2p-q+1}+\frac{p\mu}{4}\int_\Omega u^{p+1}v^{2p-q}+C\mu^{-p}\int_\Omega v^{-p(2p-q)}.
\end{align*}
Using H\"{o}lder's inequality, the boundedness of $h_1$ in \eqref{eq-iva0} and Young's inequality again, it holds that
 \begin{align*}
p\int_{\Omega}u^{p-1}v^{2p-q} h_1 
\leq&p \|h_1\|_{L^\infty(\Omega\times(0,\infty))}\int_{\Omega}u^{p-1}v^{2p-q}
\leq\frac p4\int_{\Omega}u^pv^{2p-q+1}+C\int_\Omega v^{p-q+1}.
\end{align*}
In similar manners, upon twice uses of Young's inequality it follows that
\begin{align*}
(2p-q)\int_{\Omega}u^pv^{2p-q-1}h_2\leq&(2p-q)\|h_2\|_{L^\infty(\Omega\times(0,\infty))}\int_{\Omega}u^pv^{2p-q-1}\\
\leq& \frac p4\int_{\Omega}u^pv^{2p-q+1}+\frac{p\mu}{4}\int_\Omega u^{p+1}v^{2p-q}+C\mu^{-p}\int_\Omega v^{-p(2p-q)}.
\end{align*}

Since  $q\in[p,2p)$, the $L^1$-bound of $v$ and the lower bound of $v$   from \eqref{eq-0bdd}-\eqref{eq-uL1} imply  that
 \begin{align*}
\int_\Omega v^{p+1-q}\leq C_T(1+\mu^{-1})\quad \mathrm{and}\quad \int_\Omega v^{-p(2p-q)}\leq C_T,
\end{align*}
where $
\sup_{T>0}C_T<\infty$ if \eqref{eq-iva1} is valid.

By means of Young's inequality again, we conclude that
 \begin{align*}
&p[\chi(p-1)-2(2p-q)]\int_\Omega u^{p-1}v^{2p-q-1}\nabla u\cdot\nabla v\\
\leq&\frac{p(p-1)}{2} \int_\Omega u^{p-2}v^{2p-q}|\nabla u|^2+\frac{p}{2(p-1)}[\chi(p-1)-2(2p-q)]^2\int_\Omega u^pv^{2p-q-2}|\nabla v|^2.
\end{align*}
Collecting these estimates guarantees  the desired  differential inequality \eqref{eq-ineLu}.
\end{proof}
To control the bad terms on the right hand of \eqref{eq-ineLu} in Lemma \ref{le-inequ}, we now consider the time evolution of $\int_{\Omega} \frac{|\nabla v|^{2p}}{v^q}$ for some suitable $p, q>0$.

\begin{lemma}\label{le-inequ2}
For  $p>1$ and $q\in [0, 2p-1]$,   the local-in-time solution of \eqref{eq-r01} fulfills  that
 \begin{align}\label{eq-ineLu2}
&\frac{d}{dt}\int_{\Omega} \frac{|\nabla v|^{2p}}{v^q}+\frac{2p(2p-q-1)}{q+1}\int_\Omega\frac{|\nabla v|^{2(p-1)}|D^2v|^2}{v^q}+(2p-q)\int_\Omega\frac{|\nabla v|^{2p}}{v^q}\nonumber\\
\leq&p\int_{\partial\Omega}\frac{|\nabla v|^{2(p-1)}\nabla|\nabla v|^2\cdot\nu}{v^q}dS+(2p-q)\int_\Omega\frac{u|\nabla v|^{2p}}{v^q}\nonumber\\
&+2p\int_\Omega\frac{|\nabla v|^{2(p-1)}\nabla v\cdot\nabla u}{v^{q-1}}+2p\int_\Omega\frac{|\nabla v|^{2(p-1)}\nabla v\cdot\nabla h_2}{v^{q}},\quad \forall t\in(0,T_{\max}).
\end{align}
 \end{lemma}
\begin{proof}
 Straightforward computations from \eqref{eq-r01} show that
 \begin{align*}
\frac{d}{dt}\int_{\Omega} \frac{|\nabla v|^{2p}}{v^q}=&2p\int_\Omega\frac{|\nabla v|^{2(p-1)}}{v^q}\nabla v\cdot\nabla \left(\Delta v-v+ uv+h_2\right)
\\&-q\int_\Omega\frac{|\nabla v|^{2p}}{v^{q+1}}(\Delta v-v+ uv+h_2).
\end{align*}
Recalling the point-wise identity that $2\nabla v\cdot\nabla\Delta v=\Delta|\nabla v|^2-2|D^2v|^2$, we find
 \begin{align*}
2p\int_\Omega\frac{|\nabla v|^{2(p-1)}\nabla v\cdot\nabla \Delta v}{v^q}=p\int_\Omega\frac{|\nabla v|^{2(p-1)}\Delta|\nabla v|^2}{v^q}-2p\int_\Omega\frac{|\nabla v|^{2(p-1)}|D^2v|^2}{v^q}.
\end{align*}
The integration by parts  formula enables us to  see that
 \begin{align*}
p\int_\Omega\frac{|\nabla v|^{2(p-1)}\Delta|\nabla v|^2}{v^q}=&p\int_{\partial\Omega}\frac{|\nabla v|^{2(p-1)}\nabla|\nabla v|^2\cdot\nu}{v^q}dS-p(p-1)\int_\Omega\frac{|\nabla v|^{2(p-2)}|\nabla|\nabla v|^2|^2}{v^q}\\
&+pq\int_\Omega\frac{|\nabla v|^{2(p-1)}\nabla|\nabla v|^2\cdot\nabla v}{v^{q+1}}.
\end{align*}
Similarly, we have
 \begin{align*}
-q\int_\Omega\frac{|\nabla v|^{2p}}{v^{q+1}}\Delta v=pq\int_\Omega\frac{|\nabla v|^{{2(p-1)}}\nabla|\nabla v|^2\cdot\nabla v}{v^{q+1}}-q(q+1)\int_\Omega\frac{|\nabla v|^{2(p+1)}}{v^{q+2}}.
\end{align*}
Collecting these identities together, we deduce that
 \begin{align*}
&\frac{d}{dt}\int_{\Omega} \frac{|\nabla v|^{2p}}{v^q}+2p\int_\Omega\frac{|\nabla v|^{2(p-1)}|D^2v|^2}{v^q}+
p(p-1)\int_\Omega\frac{|\nabla v|^{2(p-2)}|\nabla|\nabla v|^2|^2}{v^q}\\
&+q\int_\Omega\frac{|\nabla v|^{2p}}{v^{q+1}}h_2+(2p-q)\int_\Omega\frac{|\nabla v|^{2p}}{v^q}+q(q+1)\int_\Omega\frac{|\nabla v|^{2(p+1)}}{v^{q+2}}\\
=&p\int_{\partial\Omega}\frac{|\nabla v|^{2(p-1)}\nabla|\nabla v|^2\cdot\nu}{v^q}dS+2pq\int_\Omega\frac{|\nabla v|^{2(p-1)}\nabla|\nabla v|^2\cdot\nabla v}{v^{q+1}}
+(2p-q)\int_\Omega\frac{|\nabla v|^{2p}u}{v^q}\\
&+2p\int_\Omega\frac{|\nabla v|^{2(p-1)}\nabla v\cdot\nabla u}{v^{q-1}}+2p\int_\Omega\frac{|\nabla v|^{2(p-1)}\nabla v\cdot\nabla h_2}{v^{q}}.
\end{align*}

The Young's inequality results   the second term  as
 \begin{align*}
2pq\int_\Omega\frac{|\nabla v|^{2(p-1)}\nabla|\nabla v|^2\cdot\nabla v}{v^{q+1}}\leq& \frac{p^2q}{q+1}\int_\Omega\frac{|\nabla v|^{2(p-2)}|\nabla|\nabla v|^2|^2}{v^q}+q(q+1)\int_\Omega\frac{|\nabla v|^{2(p+1)}}{v^{q+2}}.
\end{align*}
Here, for $p>1$ we also note that if $q\in[0,p-1]$, then $\frac{p^2q}{q+1}\leq p(p-1)$, and thereby
\begin{align*}
&\frac{p^2q}{q+1}\int_\Omega\frac{|\nabla v|^{2(p-2)}|\nabla|\nabla v|^2|^2}{v^q}\leq p(p-1)\int_\Omega\frac{|\nabla v|^{2(p-2)}|\nabla|\nabla v|^2|^2}{v^q};
\end{align*}
if $q>p-1$, then $\frac{p^2q}{q+1}> p(p-1)$, and thereby
\begin{align*}
&\frac{p^2q}{q+1}\int_\Omega\frac{|\nabla v|^{2(p-2)}|\nabla|\nabla v|^2|^2}{v^q}\\
\leq&4p\left(\frac{pq}{q+1}-p+1\right)\int_\Omega\frac{|\nabla v|^{2(p-1)}|D^2v|^2}{v^q}+p(p-1)\int_\Omega\frac{|\nabla v|^{2(p-2)}|\nabla|\nabla v|^2|^2}{v^q}.
\end{align*}
Invoking these estimates  and the nonnegativity of $v$ and $h_2$, we obtain
\eqref{eq-ineLu2},  as desired.
\end{proof}

To make use of the dissipative term    in \eqref{eq-ineLu2}, we now derive the following estimate.

\begin{lemma}\label{le-inequ3}
Let $\Omega\subset \mathbb{R}^n$ be a bounded and smooth domain,  $w\in\mathcal{C}^2(\overline{\Omega})$ be such that $w>0$ in $\Omega$ and $\frac{\partial w}{\partial\nu}=0$ on $\partial\Omega$, and for any $p\geq1$ and $q>0$ let
 \begin{align}
 \label{g-def}
g(p,q,\beta_+):=\frac{n\beta_++4 p\beta_+^2}{(2 q+1)\beta_+-p},\quad \beta_{+}:=\frac{p}{2q+1}+\frac{\sqrt{4p^2+(2q+1)n}}{2(2q+1)}.
\end{align}
It holds that
 \begin{align}\label{eq-inequ}
\int_{\Omega} \frac{|\nabla w|^{2 p+2}}{w^{q+2}} \leq g(p,q,\beta_{+}) \int_{\Omega} \frac{\left|D^2 w\right|^2|\nabla w|^{2 p-2}}{w^q}.
\end{align}
 \end{lemma}

 \begin{proof}
 For any $p\geq1$, $q>0$ and $\beta>\frac{p}{2q+1}$, it follows from \cite[Theorem 2]{WangD2022M2as} that
\begin{align*}
\int_{\Omega} \frac{|\nabla w|^{2 p+2}}{w^{q+2}} \leq \frac{n\beta+4 p\beta^2}{(2 q+1)\beta-p} \int_{\Omega} \frac{\left|D^2 w\right|^2|\nabla w|^{2 p-2}}{w^q}=g(p,q,\beta) \int_{\Omega} \frac{\left|D^2 w\right|^2|\nabla w|^{2 p-2}}{w^q}.
\end{align*}
To find the minimum of $g$ defined by \eqref{g-def}, we compute  that
\begin{align*}
\frac{\partial g(p,q,\beta)}{\partial\beta}=\frac{\left[4(2q+1)\beta^2-8p\beta-n\right]p}{[(2 q+1)\beta-p]^2},
\end{align*}
and deduce that if $\beta\in (\frac{p}{2q+1}, \beta_+)$, then
 $\frac{\partial g(p,q,\beta)}{\partial\beta}<0$; if $\beta>\beta_+$, then  $\frac{\partial g(p,q,\beta)}{\partial\beta}>0$.
 This directly tells us the global minimum of $g$ is achieved at $\beta=\beta_+$:
\[
g(p,q,\beta)\geq g(p,q,\beta_{+}),\quad \forall \beta>\frac{p}{2q+1}.
\]
This is enough to ensure  \eqref{eq-inequ}.
\end{proof}

Now, a useful observation for our global existence and boundedness is  the  following estimate  on the boundary integrals appearing  in \eqref{eq-ineLu2} of Lemma \ref{le-inequ2}.

\begin{lemma}\label{le-ineqbi}
Let $w\in\mathcal{C}^2(\overline{\Omega})$ be positive   in $\overline{\Omega}$ with $\frac{\partial w}{\partial\nu}=0$ on $\partial\Omega$.
For any $p>1$ and $\kappa,q>0$, there exists $C=C(p,\kappa,q)>0$ such that
 \begin{align}\label{eq-ineqbi}
\int_{\partial\Omega}\frac{|\nabla w|^{2(p-1)}\nabla|\nabla w|^2\cdot\nu}{w^q}dS
 \leq& \kappa \int_\Omega\frac{|\nabla w|^{2(p-1)}|D^2w|^2}{w^q}+ \kappa \int_\Omega \frac{|\nabla w|^{2(p+1)}}{w^{q+2}}+C\int_\Omega w^{2p-q}.
\end{align}
 \end{lemma}

\begin{proof}
Thanks to the geometric pointwise inequality $\nabla |\nabla w|^2\cdot \nu\leq C|\nabla w|^2$ on $\partial \Omega$ for some $C=C(\Omega)>0$ (cf. \cite[Lemma 4.2]{MS2014poin}), it follows that
\[
\int_{\partial\Omega}\frac{|\nabla w|^{2(p-1)}\nabla|\nabla w|^2\cdot\nu}{w^q} dS \leq C\int_{\partial\Omega}\frac{|\nabla w|^{2p}}{w^q},
\]
which, together with the boundary trace embedding  $\|w\|_{L^1(\partial\Omega)}\leq C\|w\|_{W^{1,1}(\Omega)}$, leads to
 \begin{align*}
\int_{\partial\Omega}\frac{|\nabla w|^{2(p-1)}\nabla|\nabla w|^2\cdot\nu}{w^q} dS  \leq&C \int_{\Omega}\left|\nabla\left(\frac{|\nabla w|^{2p}}{w^q}\right)\right|+C \int_{\Omega}\frac{|\nabla w|^{2p}}{w^q}\\
\leq&C \int_{\Omega}\left|\frac{p|\nabla w|^{2(p-1)}\nabla|\nabla w|^2}{w^q}-q\frac{|\nabla w|^{2p}\nabla w}{w^{q+1}}\right|+C \int_{\Omega}\frac{|\nabla w|^{2p}}{w^q}\\
\leq&C \int_{\Omega}\frac{|\nabla w|^{2p-1}|D^2w|}{w^q}+C\int_{\Omega}\frac{|\nabla w|^{2p+1}}{w^{q+1}}+C \int_{\Omega}\frac{|\nabla w|^{2p}}{w^q}.
\end{align*}
 In light of Young's inequality, for any $\kappa_1>0$ it holds that
  \begin{align*}
&\int_{\partial\Omega}\frac{|\nabla w|^{2(p-1)}\nabla|\nabla w|^2\cdot\nu}{w^q}dS\\
\leq&\kappa_1 \int_\Omega\frac{|\nabla w|^{2(p-1)}|D^2w|^2}{w^q} +C\int_{\Omega}\frac{|\nabla w|^{2p+1}}{w^{q+1}}+C\left(1+\frac{1}{\kappa_1}\right)\int_{\Omega}\frac{|\nabla w|^{2p}}{w^q}.
\end{align*}
 Similarly, for any $\kappa_2,\kappa_3>0$ we have
   \begin{align*}
C\int_{\Omega}\frac{|\nabla w|^{2p+1}}{w^{q+1}}\leq \kappa_2 \int_\Omega \frac{|\nabla w|^{2(p+1)}}{w^{q+2}}+C(\kappa_2)\int_\Omega w^{2p-q},
\end{align*}
and
 \begin{align*}
C\left(1+\frac{1}{\kappa_1}\right) \int_{\Omega}\frac{|\nabla w|^{2p}}{w^q}\leq \kappa_3 \int_\Omega \frac{|\nabla w|^{2(p+1)}}{w^{q+2}}+C(\kappa_1,\kappa_3)\int_\Omega w^{2p-q}.
\end{align*}
 Hence, collecting these inequalities, we arrive at
  \begin{align*}
&\int_{\partial\Omega}\frac{|\nabla w|^{2(p-1)}\nabla|\nabla w|^2\cdot\nu}{w^q}dS\\
 \leq& \kappa_1 \int_\Omega\frac{|\nabla w|^{2(p-1)}|D^2w|^2}{w^q} +(\kappa_2+\kappa_3) \int_\Omega \frac{|\nabla w|^{2(p+1)}}{w^{q+2}}+C(\kappa_1,\kappa_2,\kappa_3)\int_\Omega w^{2p-q}.
\end{align*}
 Since $\kappa_i>0$, $i=1,2,3$, are arbitrary, this gives rise to the sought estimate
 \eqref{eq-ineqbi}.
\end{proof}

With Lemmas \ref{le-inequ3}  and \ref{le-ineqbi} at hand, we further   estimate for the evolution of $\int_\Omega\frac{|\nabla v|^{2p}}{v^q}$.
\begin{lemma}\label{le-inev2}
Let $g(p,q,\beta_+)$ be defined in \eqref{g-def} with $p>1$ and $q\in [0, 2p-1)$. Then for all $T>0$ there exists $C_T>0$ such that for $\widetilde{T}:=\min\{T,T_{\max}\}$ it holds that
 \begin{align}\label{eq-inev2}
&\frac{d}{dt}\int_{\Omega} \frac{|\nabla v|^{2p}}{v^q}+\frac{p(2p-q-1)}{2(q+1)g(p,q,\beta_+)}\int_\Omega\frac{|\nabla v|^{2(p+1)}}{v^{q+2}} +\frac{2p-q}{2}\int_\Omega\frac{|\nabla v|^{2p}}{v^q}\nonumber\\
\leq&\left(B_1+B_2\right)\int_\Omega u^{p+1}v^{2p-q}+C_T(1+\mu^{-1})^{2p-q},\quad \forall t\in(0,\widetilde{T}),
\end{align}
where $\sup_{T>0}C_T<\infty$ if  \eqref{eq-iva1} is valid, and $B_1, B_2$ are explicit functions of $p,q,n$:
 \begin{equation}
 \label{eq-AB}
 \begin{cases}
B_1:=B_1(p,q,n)=\frac{2}{p+1}\left\{\frac{p(p+1)(2p-q-1)}{4(q+1)(p-1)g(p,q,\beta_+)}\right\}^{-\frac{p-1}{2}}\left\{\frac{4p(q+1)(\sqrt{n}+2p-2)^2}{2p-q-1}\right\}^\frac{p+1}{2},\\[0.25cm]
B_2:=B_2(p,q,n)=\frac{(2pq-q)^{p+1}}{p+1}\left\{\frac{(p+1)(2p-q-1)}{4(q+1)g(p,q,\beta_+)}\right\}^{-p}.
 \end{cases}
\end{equation}
 \end{lemma}

\begin{proof}
We shall estimate the terms on the right-hand side of \eqref{eq-ineLu2}. Firstly, for any $p>1$ and $q\in[0,2p-1)$, by means of \eqref{eq-ineqbi},    there exists $C>0$ such that
 \begin{align*}
p\int_{\partial\Omega}\frac{|\nabla v|^{2(p-1)}\nabla|\nabla v|^2\cdot\nu}{v^q}dS
\leq&\frac{p(2p-q-1)}{4(q+1)}\int_\Omega\frac{|\nabla v|^{2(p-1)}|D^2v|^2}{v^q}+C\int_\Omega v^{2p-q}\\
&+\frac{p(2p-q-1)}{4(q+1)g(p,q,\beta_+)}\int_\Omega\frac{|\nabla v|^{2(p+1)}}{v^{q+2}}.
\end{align*}
An application of
the Gagliardo-Nirenberg inequality   and the $L^1$-bound of $v$ in   \eqref{eq-uL1} implis that
\begin{align*}
\int_\Omega v^{2p-q} =&\left\|v^{\frac{2p-q}{2(p+1)}}\right\|_{L^{2(p+1)}}^{2(p+1)}\\
\leq& C\left\|v^{\frac{2p-q}{2(p+1)}}\right\|_{L^\frac{2(p+1)}{2p-q}}^{2(p+1)(1-\theta)}\left\|\nabla v^{\frac{2p-q}{2(p+1)}}\right\|_{L^{2(p+1)}}^{2(p+1)\theta}+C\left\|v^{\frac{2p-q}{2(p+1)}}\right\|_{L^\frac{2(p+1)}{2p-q}}^{2(p+1)}\\
\leq&C(1+\mu^{-1})^{(2p-q)(1-\theta)}\left\|\nabla v^{\frac{2p-q}{2(p+1)}}\right\|_{L^{2(p+1)}}^{2(p+1)\theta}+C(1+\mu^{-1})^{2p-q},
\end{align*}
where $\theta=\frac{n(2p-1-q)}{n(2p-1-q)+2(p+1)}\in (0,1)$ due to $q<2p-1$. Then  Young's inequality entails
\begin{align*}
C\int_\Omega v^{2p-q}
\leq&\frac{p(2p-q-1)}{4(q+1)g(p,q,\beta_+)}\int_\Omega\frac{|\nabla v|^{2(p+1)}}{v^{q+2}} +C(1+\mu^{-1})^{2p-q}.
\end{align*}
Collecting these estimates, we arrive at
 \begin{align}\label{eq-gvLpq}
p\int_{\partial\Omega}\frac{|\nabla v|^{2(p-1)}\nabla|\nabla v|^2\cdot\nu}{v^q}dS
\leq&\frac{p(2p-q-1)}{4(q+1)}\int_\Omega\frac{|\nabla v|^{2(p-1)}|D^2v|^2}{v^q} \nonumber\\
&+\frac{p(2p-q-1)}{2(q+1)g(p,q,\beta_+)}\int_\Omega\frac{|\nabla v|^{2(p+1)}}{v^{q+2}} +C(1+\mu^{-1})^{2p-q}.
\end{align}
For the third term on the right-hand side of \eqref{eq-ineLu2}, upon integration by parts,  one can see that
  \begin{align*}
&2p\int_\Omega\frac{|\nabla v|^{2(p-1)}\nabla v\cdot\nabla u}{v^{q-1}}\\
=&-2p\int_\Omega\frac{u|\nabla v|^{2(p-1)}\Delta v }{v^{q-1}}-2p(p-1)\int_\Omega\frac{u|\nabla v|^{2(p-2)}\nabla v\cdot \nabla|\nabla v|^2}{v^{q-1}}+2p(q-1)\int_\Omega\frac{u|\nabla v|^{2p}}{v^{q}}\\
\leq&-2p\int_\Omega\frac{u|\nabla v|^{2(p-1)}\Delta v }{v^{q-1}}+4p(p-1)\int_\Omega\frac{u|\nabla v|^{2(p-1)}|D^2 v|}{v^{q-1}}+2p(q-1)\int_\Omega\frac{u|\nabla v|^{2p} }{v^{q}},
\end{align*}
which along with Young's inequality and the inequality $|\Delta v|\leq \sqrt{n}|D^2v|$ implies that
  \begin{align}\label{eq-gvLpq2}
2p\int_\Omega\frac{|\nabla v|^{2(p-1)}\nabla v\cdot\nabla u}{v^{q-1}}\leq&\frac{p(2p-q-1)}{4(q+1)}\int_\Omega\frac{|\nabla v|^{2(p-1)}|D^2v|^2}{v^q}  +2p(q-1)\int_\Omega\frac{u|\nabla v|^{2p}}{v^{q}}\nonumber\\
&+\frac{4p(q+1)(\sqrt{n}+2p-2)^2}{2p-q-1}\int_\Omega\frac{u^2|\nabla v|^{2(p-1)}}{v^{q-2}}.
\end{align}
Using Young's inequality again, it follows that
  \begin{align}\label{eq-gvLpq3}
&\frac{4p(q+1)(\sqrt{n}+2p-2)^2}{2p-q-1}\int_\Omega\frac{u^2|\nabla v|^{2(p-1)}}{v^{q-2}}\nonumber\\
\leq&\frac{p(2p-q-1)}{4(q+1)g(p,q,\beta_+)}\int_\Omega\frac{|\nabla v|^{2(p+1)}}{v^{q+2}}+B_1\int_\Omega u^{p+1}v^{2p-q},
\end{align}
where $B_1$ is given in \eqref{eq-AB}.
By \eqref{eq-ineLu2} and \eqref{eq-gvLpq2},  we use Young's inequality again to get that
 \begin{align}\label{eq-gvLpq4}
&\{2p(q-1)+(2p-q)\}\int_\Omega\frac{u|\nabla v|^{2p}}{v^q}\nonumber\\
\leq& \frac{p(2p-q-1)}{4(q+1)g(p,q,\beta_+)}\int_\Omega\frac{|\nabla v|^{2(p+1)}}{v^{q+2}}+B_2\int_\Omega u^{p+1}v^{2p-q},
\end{align}
where  $B_2$ is given in \eqref{eq-AB}.
Similarly, based on the regularity of $h_2$ in \eqref{eq-iva0} and the lower bound of $v$ in  \eqref{eq-0bddc} and  \eqref{eq-0bdd}, one can use  Young's inequality again to pick $C>0$ such that
 \begin{align*}
2p\int_\Omega\frac{|\nabla v|^{2(p-1)}\nabla v\cdot\nabla h_2}{v^{q}} \leq&\frac{2p-q}{2}\int_\Omega\frac{|\nabla v|^{2p}}{v^q}+C\int_\Omega\frac{|\nabla h_2|^{2p}}{v^q}\nonumber\\
\leq&\frac{2p-q}{2}\int_\Omega\frac{|\nabla v|^{2p}}{v^q} +C\|\nabla h_2\|^{2p}_{L^{2p}}\|v^{-1}\|_{L^\infty}^q\nonumber\\
\leq&\frac{2p-q}{2}\int_\Omega\frac{|\nabla v|^{2p}}{v^q} +C_T,
 \end{align*}
 where $\sup_{T>0}C_T<\infty$ if  \eqref{eq-iva1} holds.

Consequently, combining with \eqref{eq-gvLpq}, \eqref{eq-gvLpq2}, \eqref{eq-gvLpq3} and \eqref{eq-gvLpq4}, we conclude  that
 \begin{align*}
&p\int_{\partial\Omega}\frac{|\nabla v|^{2(p-1)}\nabla|\nabla v|^2\cdot\nu}{v^q}dS +(2p-q)\int_\Omega\frac{u|\nabla v|^{2p}}{v^q}\nonumber\\
&+2p\int_\Omega\frac{|\nabla v|^{2(p-1)}\nabla v\cdot\nabla u}{v^{q-1}}+2p\int_\Omega\frac{|\nabla v|^{2(p-1)}\nabla v\cdot\nabla h_2}{v^{q}}\\
\leq&\frac{p(2p-q-1)}{2(q+1)}\int_\Omega\frac{|\nabla v|^{2(p-1)}|D^2v|^2}{v^q}  +\frac{p(2p-q-1)}{(q+1)g(p,q,\beta_+)}\int_\Omega\frac{|\nabla v|^{2(p+1)}}{v^{q+2}}\\
&+\left(B_1+B_2\right)\int_\Omega u^{p+1}v^{2p-q}+\frac{2p-q}{2}\int_\Omega\frac{|\nabla v|^{2p}}{v^q} +C_T(1+\mu^{-1})^{2p-q},
\end{align*}
which, upon being substituted into \eqref{eq-ineLu2}, yields
  \begin{align*}
&\frac{d}{dt}\int_{\Omega} \frac{|\nabla v|^{2p}}{v^q}+\frac{3p(2p-q-1)}{2(q+1)}\int_\Omega\frac{|\nabla v|^{2(p-1)}|D^2v|^2}{v^q}+\frac{2p-q}{2}\int_\Omega\frac{|\nabla v|^{2p}}{v^q}\nonumber\\[0.25cm]
\leq&\frac{p(2p-q-1)}{(q+1)g(p,q,\beta_+)}\int_\Omega\frac{|\nabla v|^{2(p+1)}}{v^{q+2}} +\left(B_1+B_2\right)\int_\Omega u^{p+1}v^{2p-q}+C_T(1+\mu^{-1})^{2p-q}.
\end{align*}
 Note that $2p-q-1>0$. This along with  \eqref{eq-inequ} directly gives rise to  \eqref{eq-inev2}.
\end{proof}

\subsection{Global existence and boundedness under large logistic damping}\label{se-gbedd}

By a suitable linear combination of \eqref{eq-ineLu} and \eqref{eq-inev2}, we derive the following favorable  estimate.

 \begin{lemma}\label{le-inequi0}
Let $p>1$ and $q\in[p,2p-1)$. For all $T>0$ and $\chi>0$ there exist $C_T>0$ and $\hat\mu_0=\hat\mu_0(\chi, p,q,n)>0$ such that for $\widetilde{T}:=\min\{T,T_{\max}\}$ it holds that whenever $\mu\geq \hat \mu_0$, then
   \begin{align}\label{eq-ineLu3}
\int_{\Omega}u^{p}v^{2p-q}+\int_{\Omega} \frac{|\nabla v|^{2p}}{v^q} \leq C_T(1+\mu^{-1})^{2p-q},\quad \forall t\in(0,\widetilde{T}),
\end{align}
where $C_T>0$ is independent of $\chi$ and  $
\sup_{T>0}C_T<\infty$ if \eqref{eq-iva1} is valid.
 \end{lemma}

 \begin{proof}
Let $p>1$ and $q\in[p,2p-1)$. An application of Young's inequality gives us
\begin{align*}
&\left\{p(2p-q)\chi+\frac{p}{2(p-1)}[\chi(p-1)-2(2p-q)]^2\right\}\int_\Omega u^pv^{2p-q-2}|\nabla v|^2\\
\leq&\frac{p(2p-q-1)}{2(q+1)g(p,q,\beta_+)}\int_\Omega\frac{|\nabla v|^{2(p+1)}}{v^{q+2}} +B_3 \int_\Omega u^{p+1}v^{2p-q},
\end{align*}
where $B_3=B_3(\chi,p,q,n)$ with
\begin{align}\label{B3-def}
B_3:=\frac{p}{p+1}\left\{\frac{p(p+1)(2p-q-1)}{2(q+1)g(p,q,\beta_+)}\right\}^{-\frac{1}{p}}\left\{p(2p-q)\chi+\frac{p}{2(p-1)}[\chi(p-1)-2(2p-q)]^2\right\}^{\frac{p+1}{p}}.
\end{align}
Substituting the above estimate into  \eqref{eq-ineLu}, we obtain that
 \begin{align*}
&\frac{d}{dt}\int_{\Omega}u^{p}v^{2p-q} +(2p-q)\int_{\Omega}u^pv^{2p-q}+\frac{p(p-1)}{2}\int_\Omega u^{p-2}v^{2p-q}|\nabla u|^2\nonumber\\
\leq&\left(-\frac p2\mu-q+2p+B_3\right)\int_{\Omega}u^{p+1}v^{2p-q}+\frac{p(2p-q-1)}{2(q+1)g(p,q,\beta_+)}\int_\Omega\frac{|\nabla v|^{2(p+1)}}{v^{q+2}}+C_T(1+\mu^{-1}),
\end{align*}
where $\sup_{T>0}C_T<\infty$ if \eqref{eq-iva1} holds.
This, combining with \eqref{eq-inev2} and the fact $2q-p>1$, entails
 \begin{align*}
&\frac{d}{dt}\left\{\int_{\Omega}u^{p}v^{2p-q}+\int_{\Omega} \frac{|\nabla v|^{2p}}{v^q}\right\}+ (2p-q)\int_{\Omega}u^{p}v^{2p-q}+\frac{2p-q}{2}\int_\Omega\frac{|\nabla v|^{2p}}{v^q}\\
\leq &K\int_{\Omega}u^{p+1}v^{2p-q} +C_T(1+\mu^{-1})^{2p-q},
\end{align*}
where
$$
K:=-\frac{p\mu}{2}-q+2p+B_1+B_2+B_3
$$ with $B_1, B_2$ and $B_3$ given in \eqref{eq-AB} and \eqref{B3-def}, and
$\sup_{T>0}C_T<\infty$ if \eqref{eq-iva1} is  valid. Now, define
\begin{align}\label{hatmu0-def}
\hat\mu_0=\hat\mu_0(\chi, p,q,n):=\frac{2}{p}\left(B_1+B_2+B_3+2p-q\right).
\end{align}
According to the  definitions of $B_i$  and $g(p,q,\beta_+)$ in \eqref{eq-AB}, \eqref{B3-def} and \eqref{g-def}, one quickly finds that  $\hat\mu_0$ is explicit and  that
$K\leq0$  whenever $\mu\geq \hat \mu_0$. This along further  with $2p-q>1$ shows that
 \begin{align*}
\frac{d}{dt}\left\{\int_{\Omega}u^{p}v^{2p-q}+\int_{\Omega} \frac{|\nabla v|^{2p}}{v^q}\right\}+ \frac{1}{2}\left\{\int_{\Omega}u^{p}v^{2p-q}+\int_{\Omega} \frac{|\nabla v|^{2p}}{v^q}\right\}
\leq C_T(1+\mu^{-1})^{2p-q},
\end{align*}
which upon being multiplied by $e^{\frac12 t}$ and then integrated on $(0,t)$ entails directly  \eqref{eq-ineLu3}.  Moreover, the upper bound $C_T$ is independent of $\chi$. since  only $B_3$ involves the parameter $\chi$, which doesn't appear in the above estimate.
\end{proof}

Based on Lemma \ref{le-inequi0}, under an explicit logistic damping rate, we  next derive the spatial  $L^\infty$-bound of $v$, which is independent of   $\chi$.

\begin{lemma}\label{le-inequ4}
For $\chi>0$,  with  $B_i$  and $g(p,q,\beta_+)$ given in \eqref{eq-AB}, \eqref{B3-def} and \eqref{g-def}, define
\begin{align}\label{mu0-def}
\mu_0=\mu_0(\chi,n):=\inf\left\{\frac{2}{p}\left(B_1+B_2+B_3+2p-q\right): \  p>\frac n2,  \ q\in [p, 2p-1)\right\}.
\end{align}
Then for $T>0$ there exists $C_T>0$ such that for $\widetilde{T}:=\min\{T,T_{\max}\}$ it holds  whenever $\mu>\mu_0$ that
 \begin{align}\label{eq-ineLu4}
\|v(\cdot,t)\|_{L^\infty} \leq C_T,\quad \forall t\in(0,\widetilde{T}),
\end{align}
where  $C_T>0$ is independent of $\chi$ and  $\sup_{T>0}C_T<\infty$ if  \eqref{eq-iva1} is valid.
 \end{lemma}

\begin{proof}
Since $\mu>\mu_0$, it follows easily from \eqref{hatmu0-def} and \eqref{mu0-def} there exist  $p_0>\frac n2$ and $q_0\in [p_0, 2p_0-1)$ such that $\mu\geq \hat\mu_0(\chi, p_0, q_0,n)$. Therefore,   we  conclude from \eqref{eq-ineLu3}, \eqref{eq-0bddc} and \eqref{eq-0bdd} that for $T>0$ there exist $C_T>0$  and $\hat C_T>0$ (both $C_T$ and $\hat C_T$ are independent of $\chi$) such that for $\widetilde{T}:=\min\{T,T_{\max}\}$ it holds that
 \begin{align}\label{eq-ineLu42}
 \hat C_T \int_{\Omega}u^{p_0}\leq  \int_{\Omega}u^{p_0} v^{2p_0-q_0} \leq C_T,  \ \ \ \ \forall  t\in(0,\widetilde{T}),
\end{align}
where $\sup_{T>0}C_T<\infty$ and   $\sup_{T>0}\hat C_T<\infty$ provided that \eqref{eq-iva1} is  valid.

Then for  $p_1\in (\frac n2, p_0)$, the  H\"{o}lder  interpolation inequality shows that
\begin{align*}
\|uv\|_{L^{p_1}}\leq\|u\|_{L^{p_0}}\|v\|_{L^{\frac{p_1p_0}{p_0-p_1}}}
\leq\|u\|_{L^{p_0}}\|v\|_{L^1}^{\frac{p_0-p_1}{p_1p_0}}\|v\|_{L^\infty}^{1-\frac{p_0-p_1}{p_1p_0}},
\end{align*}
which, together with $(L^{p_0}, L^1)$-boundedness of $(u,v)$  in \eqref{eq-ineLu42} and  \eqref{eq-uL1},  leads to
 \begin{align}\label{eq-ineLu41}
\|uv\|_{L^{p_1}}\leq& C_T\|v\|_{L^\infty}^{1-\frac{p_0-p_1}{p_1p_0}},\ \ \ \  \forall  t\in(0, \widetilde{T}).
\end{align}
According to
 the variation-of-constants formula for $v$ and the well-known  estimates  of the Neumann heat semigroup (e.g., \cite[Lemma 1.3]{Winkler2010}, \cite[Lemma 2.1]{Cao2015} or \cite[Lemma 2.1]{CLX-25}),  we see that
\begin{align*}
\|v(\cdot, t)\|_{L^\infty}\leq & \|e^{t(\Delta-1)}v_{0}\|_{L^\infty}+\int_{0}^{t} \left\|e^{(t-s)(\Delta-1)}\left(uv+ h_2\right)\right\|_{L^\infty} d s\\
\leq&C+C\int_0^t\left(1+(t-s)^{-\frac{n}{2p_1}}\right)e^{-(t-s)}\|uv+h_2\|_{L^{p_1}}ds,
\end{align*}
which, combining with \eqref{eq-ineLu41} and the regularity of $h_2$ in \eqref{eq-iva0}, entails
\begin{align*}
\|v(\cdot, t)\|_{L^\infty}\leq&C+C_T\int_0^t\left(1+(t-s)^{-\frac{n}{2p_1}}\right)e^{-(t-s)}
\left(\|v(\cdot, s)\|_{L^\infty}+1\right)^{1-\frac{p_0-p_1}{p_1p_0}} ds.
\end{align*}
Due to $\frac{n}{2p_1}\in(0,1)$, for any $\widehat{T}\in(0,\widetilde{T})$, letting
\[
M(\widehat{T}):=\sup_{t\in(0,\widehat{T})}\|v(\cdot,t)\|_{L^\infty}+1,
\]
it follows that
\[
M(\widehat{T})\leq C+C_TM^{1-\frac{p_0-p_1}{p_1p_0}}(\widehat{T}),\quad \quad \widehat{T}\in(0,\widetilde{T}).
\]
Then, due to  $1-\frac{p_0-p_1}{p_1p_0}\in(0,1)$, the Young's inequality simply gives
\[
M(\widehat{T})\leq C_T,\quad \quad \forall \widehat{T}\in(0,\widetilde{T}),
\]
which on taking $\widehat{T}\nearrow  \widetilde{T}$ results in
 \begin{align*}
\|v(\cdot, t)\|_{L^\infty}
\leq &C_T,\ \ \  \forall t\in(0,\widetilde{T}),
\end{align*}
which,  together with   the fact that $\sup_{T>0}C_T<\infty$ if  \eqref{eq-iva1} is valid, concludes the bound \eqref{eq-ineLu4}.
\end{proof}

Next, for any $q\in (1,\infty]$, we proceed to show the spatial $(L^\infty, L^q)$-boundedness of $(u,\nabla v)$.
\begin{lemma}\label{le-inequi}
Under Lemma  \ref{le-inequ4}, for all $T>0$ and $q\in[1,\infty]$, there exists $C_T>0$ such that for $\widetilde{T}:=\min\{T,T_{\max}\}$ it holds that
 \begin{align}\label{eq-ineLui}
\|u(\cdot,t)\|_{L^\infty} +\|\nabla v(\cdot,t)\|_{L^q} \leq C_T,\quad \forall t\in(0,\widetilde{T}),
\end{align}
where  $
\sup_{T>0}C_T<\infty$ if \eqref{eq-iva1} is valid.
 \end{lemma}

\begin{proof}
By  \eqref{eq-ineLu4} and \eqref{eq-ineLu42},   there exists  $p_0>\frac n2$ such that
\begin{align}\label{u-lp-bdd}
\|u\|_{L^{p_0}}+\|uv+h_2\|_{L^{p_0}}\leq C_T,\quad \forall t\in(0,\widetilde{T}).
\end{align}
where $\sup_{T>0}C_T<\infty$ if \eqref{eq-iva1} holds (All the remaining $C_T$ fulfills this property, thus omitted).

 We proceed with  $\frac n2<p_0<n$ (the otherwise case is much simpler); in such case,  one can see  that
 \[
 \forall q\in \left(n, \  \frac{np_0}{n-p_0}\right)\Longrightarrow 0<\frac12+\frac n2(\frac{1}{p_0}-\frac{1}{q})<1.
 \]
 Then one can readily  infer from the smoothing effect of the Neumann semigroup $(e^{t\Delta})_{t\geq0}$  and \eqref{u-lp-bdd} that
\begin{equation}\label{eq-ineLi1}
\begin{split}
\|\nabla v(\cdot, t)\|_{L^q} \leq&\|\nabla e^{t(\Delta-1)} v_{0}\|_{L^q}+\int_{0}^{t} \left\|\nabla e^{(t-s)(\Delta-1)}\left(uv+ h_2\right)\right\|_{L^q}d s\\
\leq& C+C\int_{0}^{t} \left(1+(t-s)^{-\frac12-\frac n2(\frac{1}{p_0}-\frac{1}{q})}\right)e^{-(t-s)}\|uv+h_2\|_{L^{p_0}}ds,\\
\leq &C+C_T\int_{0}^{t} \left(1+(t-s)^{-\frac12-\frac n2(\frac{1}{p_0}-\frac{1}{q})}\right)e^{-(t-s)}ds\\
\leq & C_T, \quad \forall t\in(0,\widetilde{T}).
\end{split}
\end{equation}
For any such $q$, one can see that
\[
 \forall r>\max\{\frac{np_0}{2p_0-n}, \  p_0\}=\frac{np_0}{2p_0-n} \Longrightarrow \frac1p:=\frac1r+\frac1q<\frac1n.
 \]
By \eqref{u-lp-bdd}, \eqref{eq-ineLi1} and the lower bound of $v$, we use interpolation inequality to deduce that
\begin{align*}
 \left\|\frac uv\nabla v\right\|_{L^p}&\leq  \|u\|_{L^r}\|v^{-1}\|_{L^\infty}\|\nabla v\|_{L^q}\\
 &\leq \|u\|_{L^\infty}^\frac{r-p_0}{r}\|u\|_{L^{p_0}}^\frac{p_0}{r}\|v^{-1}\|_{L^\infty}\|\nabla v\|_{L^q}\leq C_T \|u\|_{L^\infty}^\frac{r-p_0}{r}, \quad \forall t\in(0,\widetilde{T}).
\end{align*}
 Based on this and \eqref{u-lp-bdd},  using the maximum principle and the $L^p$-$L^q$ estimates  of the Neumann heat semigroup again, we quickly infer that
\begin{align*}
\|u(\cdot, t)\|_{L^\infty} \leq&\|e^{t(\Delta-1)} u_{0}\|_{L^\infty}+\int_{0}^{t} \left\|e^{(t-s)(\Delta-1)}\left( \chi\nabla\cdot(\frac uv\nabla v) + h_1+(r+1)u\right)\right\|_{L^\infty}d s\\
\leq& C+C\int_{0}^{t} \left(1+(t-s)^{-\frac12-\frac{n}{2p}}\right)e^{-(t-s)}\left\|\frac uv\nabla v\right\|_{L^p}ds\\
&+C\int_{0}^{t} \left(1+(t-s)^{-\frac{n}{2p_0}}\right)e^{-(t-s)}\|h_1+(r+1)u\|_{L^{p_0}}ds\\
\leq & C_T+C_T\int_{0}^{t} \left(1+(t-s)^{-\frac12-\frac{n}{2p}}\right)e^{-(t-s)}\|u\|_{L^\infty}^\frac{r-p_0}{r}ds,  , \quad \forall t\in(0,\widetilde{T}).
\end{align*}
For any $T\in(0,\widetilde{T})$, setting  $S(T):=\sup\{\|u(\cdot,t)\|_{L^\infty}: \  t\in(0, T)\}$, it quickly follows  that
\[
S(T)\leq C_T+C_TS^{\frac{r-p_0}{r}}(T)\leq \frac12 S(T)+C_T,
\]
 enforcing immediately
\[
\sup\{\|u(\cdot,t)\|_{L^\infty}: \  t\in(0, T)\}=S(T)\leq C_T,\quad \quad \forall T\in(0,\widetilde{T}).
\]
Now, we can set $p_0=q=\infty$ in \eqref{eq-ineLi1}, and thus, we establish the  estimate \eqref{eq-ineLui}.
\end{proof}

Our  global existence and boundedness has  actually been completed already.

\begin{proof}[Proof of Theorem \ref{th-global}]
 In fact, combining with Lemma \ref{le-glex0},  Lemma \ref{le-inequ4} and Lemma \ref{le-inequi}, yields a contradiction of evidence  \eqref{eq-bpm},   which means that we must have $T_{\max}=\infty$ in Lemma \ref{le-lolex}. This also implies the statement on global existence described in Theorem \ref{th-global}. Finally, the
boundedness  in \eqref{eq-bedds} follows from \eqref{eq-0bdd}, \eqref{eq-ineLu4} and \eqref{eq-ineLui} directly.
\end{proof}

\begin{remark}\label{bdd-rem}
For $n\in\{2,3\}$, we  choose $p=2=q$ to compute from \eqref{eq-AB}, \eqref{B3-def} and \eqref{g-def} that
\[
B_1=
\begin{cases}
\begin{split}
64\sqrt{3}(\sqrt2+2)^3\frac{(4+\sqrt{26})}{5},\quad&\text{if}\,\,\,n=2,\\
64\sqrt{3}(\sqrt3+2)^3\frac{(4+\sqrt{31})}{5},\quad&\text{if}\,\,\,n=3,
\end{split}
\end{cases}, \  \
B_2=
\begin{cases}
\begin{split}
\frac{1152}{625}\left(\sqrt{26}+4\right)^4,\quad&\text{if}\,\,\,n=2,\\
\frac{1152}{625}\left(\sqrt{31}+4\right)^4,\quad&\text{if}\,\,\,n=3,
\end{split}
\end{cases}
\]
and
\[
B_3=
\begin{cases}
\begin{split}
\frac{2(4+\sqrt{26})}{15}\left[4\chi+(\chi-4)^2\right]^\frac32,\quad&\text{if}\,\,\,n=2,\\
\frac{2(4+\sqrt{31})}{15}\left[4\chi+(\chi-4)^2\right]^\frac32,\quad&\text{if}\,\,\,n=3,
\end{split}
\end{cases}
\]
so that we have
\[
 \frac{2}{p}\left(B_1+B_2+B_3+2p-q\right)\approx \tilde\mu_0(\chi,n)=:
 \begin{cases}
\begin{split}
20666.6+1.2132\left[(\chi-2)^2+12\right]^\frac32,\quad \text{if}\,\,\,n=2,\\
26477.2+1.2757\left[(\chi-2)^2+12\right]^\frac32,\quad\text{if}\,\,\,n=3.
\end{split}
\end{cases}
\]
Then, by \eqref{mu0-def} of Lemma \ref{le-inequ4}, we see that $\mu_0(\chi,n)\leq \tilde\mu_0(\chi,n)$, and therefore, thanks to  Lemma \ref{le-inequ4} and Lemma \ref{le-inequi}, when $n=2,3$ and $\mu>\hat\mu_0(\chi,n)$, the local classical solution of the IBVP  \eqref{eq-r01} exists globally and is bounded provided that \eqref{eq-iva1} is satisfied.
 \end{remark}

\section{Longtime behaviors of bounded classical solutions}\label{se-ape}

In this section, we shall always assume that $(u,v)$ is a global classical solution of \eqref{eq-r01}, which is bounded in the sense of \eqref{eq-bedds}. Then for any $p>1$, one has, for some $C_p>0$, that
\begin{align}\label{ulp-diff-bdd}
\frac1p\frac{d}{dt}\int_\Omega u^p \leq \frac{\chi^2}{4}(p-1)\int_\Omega u^pv^{-2}|\nabla v|^2 +\int_\Omega u^{p-1}(ru+h_1)\leq C_p.
\end{align}
To get the global stability of bounded solutions, we first recall  the following  elementary  fact. 
\begin{lemma}\label{le-ODId2}
Let $\tau\geq 0, f(t)\geq0$ and $\int_\tau^\infty f(t)dt<\infty$, and, for some $a>0$,
\begin{equation}\label{lip-os}
f(t)-f(s)\leq a(t-s), \ \ \  \forall t\geq s\geq \tau.
\end{equation}
Then
\begin{equation*}
f(t)\rightarrow0 \quad\textrm{as}\quad t\rightarrow \infty.
\end{equation*}
\end{lemma}

\subsection{Stability of the semi-trivial steady state $(0,b)$.}

Now, we deduce  the stability of the semi-trivial steady state.

\begin{lemma}\label{le-UL1de}
Let \eqref{eq-iva2} hold and $r\leq \eta_0$ with $\eta_0$ being the lower bound of $v$ given in \eqref{eq-0bddc}.  Then
 \begin{align}\label{eq-1ul1e0}
\int_\Omega  u^2(\cdot,t) \rightarrow0\quad\mathrm{as}\,\,\, t\rightarrow\infty.
\end{align}
\end{lemma}

\begin{proof}
Since $v\geq \eta_0$, we integrate the $u$-equation in \eqref{eq-r01} to find that
\begin{align*}
\frac{d}{dt}\int_\Omega u + \eta_0 \int_\Omega u +\mu\int_\Omega u^2 \leq r\int_\Omega u  dx+\int_\Omega h_1,
\end{align*}
which along with the condition  $r\leq \eta_0$ shows  that
\begin{align*}
\frac{d}{dt}\int_\Omega u +\mu\int_\Omega u^2 \leq  \int_\Omega h_1.
\end{align*}
Upon integration,  this along with  \eqref{eq-iva2} shows
 \begin{align}\label{eq-il1e0}
\mu\int_0^\infty \int_\Omega u^2 \leq \int_\Omega u_0 +\int_0^\infty\int_\Omega h_1<\infty,
\end{align}
yielding directly  \eqref{eq-il1e0}.  On the other hand, by \eqref{ulp-diff-bdd}, the mean value theorem shows that $f(t):=\int_\Omega u^2$ satisfies the one-sided Lipschitz condition  \eqref{lip-os}. So Lemma \ref{le-ODId2} simply  shows \eqref{eq-1ul1e0}.
\end{proof}

A straightforward consequence  of Lemma \ref{le-UL1de} is the following
decay for  the component $v$.

\begin{lemma}\label{le-VgL2} Under  Lemma \ref{le-UL1de}, it holds that
 \begin{align}\label{eq-VgL2d}
\int_\Omega  (v(\cdot,t)-b)^2  \rightarrow0\quad\mathrm{as}\,\,\, t\rightarrow\infty.
\end{align}
\end{lemma}

\begin{proof}
Testing the first equation of \eqref{eq-r01} by $2(v-b)$ and  using Young's inequality, we have
 \begin{align*}
\frac{d}{dt}\int_\Omega(v-b)^2+\int_\Omega (v-b)^2 \leq 2\int_\Omega (h_2-b)^2+2\int_\Omega u^2v^2.
\end{align*}
This along with boundedness of solutions first ensures that $f'(t):=\frac{d}{dt}\int_\Omega(v-b)^2$ is bounded from above, and then,  by \eqref{eq-iva2} and \eqref{eq-il1e0}, upon integration, it further shows
 \begin{align*}
\int_0^\infty\int_\Omega(v-b)^2\leq \int_\Omega(v_0-b)^2+2\int_0^\infty(h_2-b)^2+2\|v\|_{L^\infty(\Omega\times (0, \infty)}^2\int_\Omega u^2<\infty.
\end{align*}
An easy use of  Lemma \ref{le-ODId2} entails  \eqref{eq-VgL2d}.
\end{proof}

Based on  the $(L^2, L^2)$-convergence of $(u,v-b)$ in \eqref{eq-1ul1e0} and  \eqref{eq-VgL2d}, we next employ  interpolation inequality and  semigroup estimate to establish  the $ W^{1,\infty}$-convergence of $v-b$ as described in \eqref{eq-smdeu} of  Theorem \ref{th-long}.

\begin{lemma}\label{le-smdeu0}
 Under  Lemma \ref{le-UL1de}, it holds that
\begin{align}\label{eq-smdeu0}
\|v(\cdot,t)-b\|_{W^{1,\infty}}\rightarrow0, \quad \mathrm{as}\,\,\, t\rightarrow\infty.
\end{align}
\end{lemma}

\begin{proof}
To start off, we rewrite the $v$-equation as
\[(v-b)_t=\Delta v-(v-b)+uv+h_2-b,
\]
 and then, we use the well-known  Neumann heat semigroup estimates to deduce that
 \begin{align*}
\| v(\cdot,t)-b\|_{W^{1,\infty}}
&\leq\| e^{t(\Delta-1)}(v_0-b)\|_{W^{1,\infty}}+ \int_0^t\left\|e^{(t-s)(\Delta-1)} \left(uv+h_2-b\right)\right\|_{W^{1,\infty}}ds\\
&\le Ce^{-t}\|v_0-b\|_{W^{1,\infty}}+C\int_{0}^t\left(1+(t-s)^{- \frac{1}{2}-\frac{1}{8}}\right)e^{-(t-s)}\left\|uv+h_2-b\right\|_{L^{4n}}ds.
\end{align*}
 H\"{o}lder's inequality, combined with the boundedness \eqref{eq-bedds},  \eqref{eq-iva0}  and interpolation inequality, entails
\begin{align*}
\left\|uv+h_2-b\right\|_{L^{4n}}\leq & \|v\|_{L^\infty}\|u\|_{L^{4n}}+\|h_2-b\|_{L^{4n}}\\
\leq&\|v\|_{L^\infty}\|u\|_{L^2}^{\frac{1}{2n}}\|u\|_{L^\infty}^{1-\frac{1}{2n}}+\|h_2-b\|_{L^2}^{\frac{1}{2n}}\|h_2-b\|_{L^\infty}^{1-\frac{1}{2n}}\\
\leq& C\left(\|u\|_{L^2}^{\frac{1}{2n}}+\|h_2-b\|_{L^2}^{\frac{1}{2n}}\right).
\end{align*}
By the boundedness   \eqref{eq-bedds}  and \eqref{eq-iva0} again,  we have
 \begin{align*}
V_{1}&:= \int_{0}^\frac t2\left(1+(t-s)^{-\frac58}\right)e^{-(t-s)}(\|u\|_{L^2}^{\frac{1}{2n}}+\|h_2-b\|_{L^2}^{\frac{1}{2n}})ds\\
&\leq C\int_{\frac t2}^{t}\left(1+\tau^{-\frac58}\right)e^{-\tau}d\tau\to 0,  \ \   \mathrm{as} \ \  t\to \infty.
\end{align*}
Thanks to  \eqref{eq-1ul1e0} and \eqref{eq-iva3}, we see from H\"{o}lder's inequality  that
\begin{align*}
V_2:=&\int_{\frac t2}^t\left(1+(t-s)^{-\frac58}\right)e^{-(t-s)}(\|u\|_{L^2}^{\frac{1}{2n}}+\|h_2-b\|_{L^2}^{\frac{1}{2n}})ds\\
\leq& \sup\limits_{s\in (\frac t2,t)}\|u(\cdot,s)\|_{L^2}^\frac{1}{2n}\int_0^{\frac t2}\left(1+\tau^{-\frac58}\right)e^{-\tau}d\tau\\
&+C\left\{\int_0^{\frac t2}\left(1+\tau^{-\frac{5n}{2(4n-1)}}\right)e^{-\tau}ds\right\}^{\frac{4n-1}{4n}}\left\{\int_{\frac t2}^t\|h_2-b\|_{L^2}^2ds\right\}^\frac{1}{4n},\\
\leq&C\sup\limits_{s\in (\frac t2,t)}\|u(\cdot,s)\|_{L^2}^\frac{1}{2n}+C\left\{\int_{\frac t2}^t\int_\Omega (h_2-b)^2dxds\right\}^\frac{1}{2n}\to 0,  \ \   \mathrm{as} \ \  t\to \infty.
\end{align*}
Collecting these verifies   \eqref{eq-smdeu0}, as desired.
\end{proof}

Similar to Lemma \ref{le-smdeu0}, we proceed to show the $L^\infty$-convergence  of the $u$-solution component.

\begin{lemma}\label{le-usmdeu0}
Under  Lemma \ref{le-UL1de}, it holds that
\begin{align}\label{eq-usmdeu0}
\|u(\cdot,t)\|_{L^\infty}\rightarrow0, \quad \mathrm{as}\,\,\, t\rightarrow\infty.
\end{align}
\end{lemma}

\begin{proof}
Using the widely known properties the Neumann heat semigroup, it follows that
\begin{align*}
\|u(\cdot, t)\|_{L^\infty}\leq& Ce^{-t}+C\int_{0}^{t} \left(1+(t-s)^{-\frac12}\right)e^{-(t-s)}\left\|\frac uv\nabla v\right\|_{L^\infty}ds\\
&+C\int_{0}^{t} \left(1+(t-s)^{-\frac{1}{4}}\right)e^{-(t-s)}\|h_1+(r+1)u\|_{L^{2n}}ds.
\end{align*}
The combined boundedness in  \eqref{eq-bedds} bounds the second term as
\begin{align*}
\left\|\frac uv\nabla v\right\|_{L^\infty}\leq\|u\|_{L^\infty}\|v^{-1}\|_{L^\infty}\|\nabla v\|_{L^\infty}\leq C \|\nabla v\|_{L^\infty}
\end{align*}
Similarly, recalling the regularity of $h_1$ in  \eqref{eq-iva0}, it shows that
\begin{align*}
\|h_1+(r+1)u\|_{L^{2n}}\leq \|h_1\|_{L^1}^\frac{1}{2n}\|h_1\|_{L^\infty}^{1-\frac{1}{2n}}+(r+1)\|u\|_{L^2}^\frac{1}{n}\|u\|_{L^\infty}^{1-\frac{1}{n}}\leq C\|h_1\|_{L^1}^\frac{1}{2n}+C\|u\|_{L^2}^\frac1n.
\end{align*}
Invoking these, we arrive at
\begin{align*}
\|u(\cdot, t)\|_{L^\infty}\leq& Ce^{-t}+C\int_{0}^{t} \left(1+(t-s)^{-\frac12}\right)e^{-(t-s)}\|\nabla v\|_{L^\infty}ds\\
&+C\int_{0}^{t} \left(1+(t-s)^{-\frac{1}{4}}\right)e^{-(t-s)}\left(\|u\|_{L^2}^\frac1n+\|h_1\|_{L^1}^\frac{1}{2n}\right)ds.
\end{align*}
By  recalling  the  spacetime integrability of $h_1$ in \eqref{eq-iva2} and the convergence in \eqref{eq-1ul1e0} and \eqref{eq-smdeu0},  a slight adaptation of the proof of Lemma \ref{le-smdeu0} gives us the desired decay
\eqref{eq-usmdeu0}.
\end{proof}

\subsection{Stability of the positive steady state $(u_*,v_*)$.}

In the sequel, we will show the stability of the positive steady state  by investigating  the time evolution of
the following functional:
\begin{equation}\label{eq-entrf}
\mathcal{E}_u(t):= \int_{\Omega} u-u_*-u_*\ln \frac{u}{u_*}, \quad t>0.
\end{equation}
Then simple analysis shows that $\mathcal{E}_u(t)\geq 0$  and that $\mathcal{E}_u(t)=0$ if and only if $u=u_*$.

\begin{lemma}\label{le-Uulnu}
Let $(u_*, v_*)$ be the steady state of  \eqref{eq-r01} defined by \eqref{eq-ssse}. Then
 \begin{align}\label{eq-Uulnu}
\mathcal{E}'_u(t) +\mu\int_\Omega \left(u-u_*\right)^2
\leq&\frac{\chi^2 u_*}{4}\int_\Omega |\nabla\ln v|^2-\int_\Omega \left(u-u_*\right) \left(v-v_*\right)+\int_\Omega  h_1.
\end{align}
\end{lemma}

\begin{proof}
Multiplying the first equation in \eqref{eq-r01} by $(1-\frac{u_*}{u})$, integrating
by parts and using the fact $r=v_*+\mu u_*$ in \eqref{eq-ssse}, we find that
\begin{align*}
\mathcal{E}'_u(t)+u_*\int_\Omega\frac{|\nabla u|^2}{u^2}=&\chi u_*\int_\Omega \frac{\nabla u}{u}\cdot \nabla\ln v +\int_\Omega  \left(1-\frac{u_*}{u}\right)h_1\\
&-\int_\Omega \left(u-u_*\right)(v-v_*)-\mu\int_\Omega  \left(u-u_*\right)^2.
\end{align*}
By  Young's inequality, it follows that
\begin{align*}
\chi u_*\int_\Omega \frac{\nabla u}{u}\cdot \nabla\ln v \leq u_*\int_\Omega\frac{|\nabla u|^2}{u^2}+\frac{u_*\chi^2}{4}\int_\Omega |\nabla \ln v|^2.
\end{align*}
Combining this estimate and dropping the obvious nonnegative term, we  conclude  \eqref{eq-Uulnu}.
\end{proof}

To absorb the first term on the right hand of \eqref{eq-Uulnu}, we estimate a similar functional:
\begin{equation}\label{eq-ventrf}
\mathcal{E}_v(t):= \int_{\Omega} v-v_*-v_*\ln \frac{v}{v_*}, \quad t>0.
\end{equation}
Similarly, it holds that $\mathcal{E}_v(t)\geq0$ and that $\mathcal{E}_v(t)=0$ if and only if $v=v_*$.
\begin{lemma}\label{le-vUulnu}
Let $(u_*, v_*)$ be the steady state of  \eqref{eq-r01} defined by \eqref{eq-ssse}. Then
 \begin{align}\label{eq-vUulnu}
&\mathcal{E}'_v(t) +v_*\int_\Omega |\nabla \ln v|^2+(1-u_*)\int_\Omega  \left(1-\frac{v_*}{v}\right)(v-v_*)\nonumber\\
=&\int_\Omega  \left(v-v_*\right)(u-u_*)+\int_\Omega \left(1-\frac{v_*}{v}\right) \left(h_2-b\right).
\end{align}
\end{lemma}

\begin{proof}
Similar to the proof of Lemma \ref{le-Uulnu}, we multiply  the second equation in \eqref{eq-r01} by $(1-\frac{v_*}{v})$ and  integrate
by parts to get
\begin{align*}
\mathcal{E}'_v(t)+v_*\int_\Omega |\nabla \ln v|^2=\int_\Omega \left(1-\frac{v_*}{v}\right) \left(-v+u v +b\right)+\int_\Omega \left(1-\frac{v_*}{v}\right) \left(h_2-b\right).
\end{align*}

Note from \eqref{eq-ssse}  that $b=v_*-u_*v_*$, and so
\begin{align*}
&\int_\Omega \left(1-\frac{v_*}{v}\right) \left(-v+u v +b\right)=-(1-u_*)\int_\Omega  \left(1-\frac{v_*}{v}\right)(v-v_*) +\int_\Omega  \left(v-v_*\right)(u-u_*).
\end{align*}
Substituting this into the above identity, we achieve \eqref{eq-vUulnu}.
\end{proof}

Combining Lemmas  \ref{le-Uulnu} and  \ref{le-vUulnu}, we now derive the $L^2$-convergence of $(u-u_*, v-v_*)$.

\begin{lemma}\label{le-dUulnu}
Under Theorem \ref{th-long} (ii), it holds that
 \begin{align}\label{eq-dUulnu}
\int_\Omega \left(u-u_*\right)^2+\int_\Omega \left(v-v_*\right)^2\rightarrow0, \quad  \mathrm{as}\,\,\, t\rightarrow\infty.
\end{align}
\end{lemma}

\begin{proof}
With the explicit formula for $(u_*, v_*)$ in \eqref{eq-sss} and  the definition of $\hat\mu_1(\chi)$ in \eqref{hatmu-def}, using the assumption $\mu>\hat\mu_1(\chi)$, upon several algebraic  computations, we find  that
$$
\frac{\chi^2u_*}{4}\leq v_*\Longleftrightarrow \frac{\chi^2}{2}(r-b)\leq r^2+r\sqrt{(\mu-r)^2+4b\mu}-(r-2b)\mu.
$$ Then an immediate  combination of \eqref{eq-Uulnu} and \eqref{eq-vUulnu} yields
 \begin{align*}
&\left(\mathcal{E}_u(t)+\mathcal{E}_v(t)\right)'+\mu \int_\Omega \left(u-u_*\right)^2 +(1-u_*)\int_\Omega  \left(1-\frac{v_*}{v}\right)(v-v_*) \nonumber\\
\leq&\int_\Omega  h_1+\int_\Omega \left(1-\frac{v_*}{v}\right) \left(h_2-b\right).
\end{align*}
Since $v_*>0$,  it follows from \eqref{eq-ssse} that $1-u_*>0$. Hence, Young's inequality gives
 \begin{align*}
\int_\Omega \left(1-\frac{v_*}{v}\right) \left(h_2-b\right)
\leq&\frac{1-u_*}{2}\int_\Omega  \left(1-\frac{v_*}{v}\right)(v-v_*)+\frac{1}{2(1-u_*)}\int_\Omega \frac{\left(h_2-b\right)^2}{v}.
\end{align*}
Due to \eqref{eq-bedds} and  \eqref{eq-0bdd}, it follows for some $M>0$ that  $M\geq v\geq \eta_0$; these enable us to deduce that
 \begin{align*}
\left(\mathcal{E}_u(t)+\mathcal{E}_v(t)\right)'+\mu \int_\Omega \left(u-u_*\right)^2+\frac{(1-u_*)}{2M}\int_\Omega  \left(v-v_*\right)^2
\leq&\int_\Omega h_1+\frac{1}{2(1-u_*)\eta_0}\int_\Omega \left(h_2-b\right)^2.
\end{align*}
Integrating this differential inequality, using  the non-negativity of $\mathcal{E}_u$ and $\mathcal{E}_v$, and the spacetime integrability of $h_1$ and $h_2$, we deduce that
\begin{equation}\label{uv-lt-nonc}
\begin{split}
&\mu\int_1^\infty \int_\Omega \left(u-u_*\right)^2 +\frac{(1-u_*)}{2M}\int_1^\infty\int_\Omega  \left(v-v_*\right)^2\\
\leq&\mathcal{E}_u(1)+\mathcal{E}_v(1)+\int_1^\infty\int_\Omega h_1+\frac{1}{2(1-u_*)\eta_0}\int_1^\infty\int_\Omega \left(h_2-b\right)^2 <\infty.
\end{split}
\end{equation}
On the other hand,  for some $C>0$,  we infer from the $v$ equation that
$$
\frac{d}{dt}\int_\Omega  \left(v-v_*\right)^2  +2\int_\Omega |\nabla v|^2=2\int_\Omega (v-v_*)(-v+uv+h_2)\leq C.
$$
Based on  this and  \eqref{ulp-diff-bdd}, a direct application of Lemma \ref{le-ODId2} to \eqref{uv-lt-nonc} entails  \eqref{eq-dUulnu}.
\end{proof}

With Lemma \ref{le-dUulnu} at hand, slightly modifying Lemmas \ref{le-smdeu0} and \ref{le-usmdeu0}, we quickly  obtain  the $(L^\infty, W^{1,\infty})$ convergence of $(u_*, v_*)$  as announced in Theorem \ref{th-long}.
\begin{lemma}\label{le-pusmde}
Under Theorem \ref{th-long} (ii), it holds that
\begin{align}\label{eq-pusmde}
\|u(\cdot,t)-u_*\|_{L^\infty}+\|v(\cdot,t)-v_*\|_{W^{1,\infty}}\rightarrow0, \quad \mathrm{as}\,\,\, t\rightarrow\infty.
\end{align}
\end{lemma}

\begin{proof}
The fact that $b=v_*-u_*v_*$ due to \eqref{eq-ssse} allows us to rewrite the $v$-equation as
\[
(v-v_*)_t=\Delta(v-v_*)-(v-v_*)+v(u-u_*)+u_*(v-v_*)+h_2-b.
\]
By the properties of the Neumann heat semigroup, it follows that
 \begin{align*}
\| v(\cdot,t)-v_*\|_{W^{1,\infty}}
&\le Ce^{-t}\|v_0-v_*\|_{W^{1,\infty}}\\
&+C\int_{0}^t\left(1+(t-s)^{-\frac58}\right)e^{-(t-s)}\left\|v(u-u_*)+u_*(v-v_*)+h_2-b\right\|_{L^{4n}}ds,
\end{align*}
The boundedness in \eqref{eq-bedds} and the regularity of $h_2$ in \eqref{eq-iva0}  and interpolation inequality imply
\begin{align*}
&\left\|v(u-u_*)+u_*(v-v_*)+h_2-b\right\|_{L^{4n}}\\
\leq&\|v\|_{L^\infty}\|u-u_*\|_{L^2}^{\frac{1}{2n}}\|u-u_*\|_{L^\infty}^{1-\frac{1}{2n}}
+u_*\|v-v_*\|_{L^2}^{\frac{1}{2n}}\|v-v_*\|_{L^\infty}^{1-\frac{1}{2n}}+\|h_2-b\|_{L^2}^{\frac{1}{2n}}\|h_2-b\|_{L^\infty}^{1-\frac{1}{2n}}\\
\leq&C\left(\|u-u_*\|_{L^2}^{\frac{1}{2n}}+\|v-v_*\|_{L^2}^{\frac{1}{2n}}+\|h_2-b\|_{L^2}^{\frac{1}{2n}}\right),
\end{align*}
and hereby  telescoping
 \begin{align*}
&\| v(\cdot,t)-v_*\|_{w^{1,\infty}}\\
&\le Ce^{-t}+C\int_{0}^t\big(1+(t-s)^{- \frac58}\big)e^{-(t-s)}\left(\|u-u_*\|_{L^2}^{\frac{1}{2n}}+\|v-v_*\|_{L^2}^{\frac{1}{2n}}+\|h_2-b\|_{L^2}^{\frac{1}{2n}}\right)ds.
\end{align*}
Then, given the $(L^2, L^2)$-convergence of $(u-u_*, v-v_*)$ in  Lemma \ref{le-dUulnu} and the spacetime integrability of $h_2$ in \eqref{eq-iva3}, we can proceed along the lines of the proof of Lemma \ref{le-smdeu0} to  get
\begin{align}\label{eq-pssde}
\|v(\cdot,t)-v_*\|_{W^{1,\infty}}\rightarrow0,\quad  \mathrm{as}\,\,\, t\rightarrow\infty.
\end{align}

On the other hand,  we use the relation  $r=v_*+\mu u_*$  due to \eqref{eq-ssse} to rewrite the $u$-equation as
\[
(u-u_*)_t= \Delta (u-u_*)-\chi\nabla\cdot\left(u\nabla\ln v\right)-u(v-v_*)-\mu u(u-u_*)+h_1.
\]
So,  the properties of the Neumann heat semigroup show that
\begin{align*}
&\|u(\cdot, t)-u_*\|_{L^\infty}\\
\leq& Ce^{-t}+C\int_{0}^{t} \left(1+(t-s)^{-\frac12}\right)e^{-(t-s)}\left\|\frac uv\nabla v\right\|_{L^\infty}ds\\
&+C\int_{0}^{t} \left(1+(t-s)^{-\frac14}\right)e^{-(t-s)}\|u-u_*-u(v-v_*)-\mu u(u-u_*)+h_1\|_{L^{2n}}ds.
\end{align*}
As before,  the interpolation inequality, \eqref{eq-iva0} and \eqref{eq-bedds} imply  that
\begin{align*}
\|u-u_*-u(v-v_*)-\mu u(u-u_*)+h_1\|_{L^{2n}}\leq  C\left(\|u-u_*\|_{L^2}^\frac{1}{n}+\|v-v_*\|_{L^2}^\frac{1}{n}+\|h_1\|_{L^1}^\frac{1}{2n}\right)
\end{align*}
and that
\begin{align*}
\left\|\frac uv\nabla v\right\|_{L^\infty}\leq \|u\|_{L^\infty}\|v^{-1}\|_{L^\infty}\|\nabla v\|_{L^\infty}\leq C\|\nabla v\|_{L^\infty}.
\end{align*}
Collecting these, we arrive at
\begin{align*}
&\|u(\cdot, t)-u_*\|_{L^\infty}\\
\leq& Ce^{-t}+C\int_{0}^{t} \left(1+(t-s)^{-\frac12}\right)e^{-(t-s)}\|\nabla v\|_{L^\infty}ds\\
&+C\int_{0}^{t} \left(1+(t-s)^{-\frac{1}{4}}\right)e^{-(t-s)}\left(\|u-u_*\|_{L^2}^{\frac{1}{n}}+\|v-v_*\|_{L^2}^\frac{1}{n}+\|h_1\|_{L^1}^\frac{1}{2n}\right)ds.
\end{align*}
Consequently, the  $L^\infty$-convergence of $\nabla v$  in \eqref{eq-pssde},  the $(L^2, L^2)$-convergence of $(u-u_*, v-v_*)$ in  Lemma \ref{le-dUulnu} and the spacetime integrability of $h_1$ in \eqref{eq-iva2}, we can proceed along the lines of the proof of Lemma \ref{le-usmdeu0} to obtain
\begin{align*}
\|u(\cdot,t)-u_*\|_{L^\infty}\rightarrow0,\quad  \mathrm{as}\,\,\, t\rightarrow\infty.
\end{align*}
This together with \eqref{eq-pssde} gives rise to \eqref{eq-pusmde}.
\end{proof}

Stability statements  in Theorem \ref{th-long} now become  immediate.

\begin{proof}[Proof of Theorem \ref{th-long}]
The stability of $(0,b)$ described in Theorem \ref{th-long}  follows from  Lemmas \ref{le-smdeu0} and  \ref{le-usmdeu0}; while the stability of  $(u_*, v_*)$ has been exactly shown in   Lemma \ref{le-pusmde}.
\end{proof}

\section{Global existence and smoothness of generalized solution}\label{se-esm}

In fact,  the generalized solvability was established by Heihoff  in the two-dimensional setting (\cite{Heihoff2020zamp}), where
$u\in L^2(\Omega\times(0,\infty))$  ensuring $v(\cdot,t)\in L^q(\Omega）$ for any $q<\infty$  played a crucial role in the analyses. However, this is unavailable in the $n$-dimensional settings as long as $n\geq3$. Hence, we first need to find
an appropriate generalized framework and  specify the concept of a global generalized super-solution, which is inspired by \cite{LIXIE2022,TW2025SCM,Lankeit2017} in some sense.

\begin{definition}\label{def-gs}
Let $r\in\R$, $\chi>0$ and $\mu>0$.
A pair $(u,v)$ is called a global generalized super-solution to  the IBVP \eqref{eq-r01} if for  any $T>0$

\noindent (1)\quad it holds that
\begin{equation}\label{eq-wdef1}
\left\{
\begin{split}
&u\in L^2(\Omega\times(0,T)),\,\,\,\nabla\ln(1+u)\in L^2(\Omega\times(0,T)),\\
&v\in L^q(\Omega\times(0,T))\quad \mathrm{with\,\,some} \,\,q>1,\\
&uv\in L^1(\Omega\times(0,T)),\,\,\,\nabla\ln v\in L^2(\Omega\times(0,T)),\\
&u(x,t)\geq0,\,\,\, v(x,t)>0,\,\,\, a.e.\,\, \mathrm{on}\,\,\Omega\times[0,T];
 \end{split}
  \right.
  \end{equation}
\noindent (2)\quad    it holds (by spacetime integration) that
\begin{align}\label{eq-wdef1-1}
\int_\Omega u+ \int_0^t\int_\Omega uv+\mu\int_0^t\int_\Omega u^2\leq\int_\Omega u_0+r\int_0^t\int_\Omega u+\int_0^t\int_\Omega h_1,\,\,\, a.e.\,\, \mathrm{in}\,\, [0,T],
  \end{align}
 and
  \begin{equation}\label{eq-wdef1-2}
  \begin{split}
&\int_\Omega (u+v)+\int_0^t\int_\Omega v+\mu\int_0^t\int_\Omega u^2\\
&\leq \int_\Omega (u_0+v_0)+r\int_0^t\int_\Omega u+\int_0^t\int_\Omega (h_1+  h_2),\,\,\, a.e.\,\, \mathrm{in}\,\, [0,T];
\end{split}
 \end{equation}
\noindent (3)\quad    it holds (by test procedure) that for each nonnegative $\varphi \in \mathcal{C}_0^\infty(\overline{\Omega}\times[0,T))$,
\begin{equation}\label{eq-wdef2}
\begin{split}
&-\int_\Omega\varphi|_{t=0}\ln (u_0+1)-\int_0^T\int_\Omega \ln(u+1) \varphi_t\\
&-\frac{\chi^2}{4}\left(\int_\Omega\ln v_0\varphi|_{t=0}+\int_0^T\int_\Omega \ln v\varphi_t\right) \\
\ge& \int_0^T\int_\Omega\left(\frac{|\nabla u|^2\varphi}{(1+u)^2}- \frac{\nabla u\cdot\nabla\varphi}{1+u}+\chi\frac{u\nabla v\cdot\nabla\varphi}{(1+u)v}-\chi\frac{u\varphi\nabla u\cdot\nabla v}{(1+u)^2v}\right)\\
&+\int_0^T\int_\Omega \frac{\varphi}{1+u}\left(-uv+ru-\mu u^2+h_1\right)\\
&+\frac{\chi^2}{4}\int_0^T\int_\Omega\left(\frac{|\nabla v|^2\varphi}{v^2}-\frac{\nabla v\cdot\nabla\varphi}{v}\right)+\frac{\chi^2}{4}\int_0^T\int_\Omega \frac{\varphi}{v}\left(-v+uv+h_2\right);
\end{split}
\end{equation}
\noindent (4)\quad   it holds (by test procedure) that for each nonnegative $\varphi \in \mathcal{C}_0^\infty(\overline{\Omega}\times[0,T))$,
\begin{align}\label{eq-wdef4}
-\int_\Omega\ln v_0\varphi|_{t=0}-\int_0^T\int_{\Omega}\ln v\varphi_t\geq
\int_0^T\int_{\Omega}\left(\frac{|\nabla v|^2\varphi}{v^2}-\frac{\nabla v\cdot\nabla\varphi}{v}-\varphi+u\varphi+\frac{\varphi h_2}{v}\right).
\end{align}
 \end{definition}
\begin{remark}
  As well-established in the literature (e.g., \cite{Lankeit2017,Winkler2015SIAML}), if $(u,v)$ is a global generalized solution in the above sense that additionally fulfills $(u,v)\in \left(\mathcal{C}^{2,1}(\bar{\Omega}\times (0, \infty))\right)^2$,then  $(u,v)$  actually
solves \eqref{eq-r01}
 classically on $\Omega\times (0, \infty)$.
\end{remark}

\subsection{Global existence  of generalized solution}
Next, we aim to prove  the
  generalized solvability of the IBVP \eqref{eq-r01}  by an approximation procedure. Accordingly, for each $\varepsilon\in(0, 1)$,  we study the following approximate problem
\begin{equation}\label{eq-app}
\left\{
\begin{split}
&u_{\varepsilon t}=\Delta u_{\varepsilon}-\chi\nabla \cdot\left(u_{\varepsilon}\nabla\ln v_{\varepsilon}\right)- u_\varepsilon v_\varepsilon+ru_\varepsilon-\mu u_\varepsilon^2+h_1, & x \in \Omega, \,t>0,\\
 &v_{\varepsilon t}=\Delta v_{\varepsilon}-v_\varepsilon+\frac{u_\varepsilon v_\varepsilon}{1+\varepsilon u_\varepsilon v_\varepsilon}+h_2, & x \in \Omega,\, t>0,\\
  &\frac{\partial u_{\varepsilon}}{\partial \nu}=\frac{\partial v_{\varepsilon}}{\partial \nu}=0, & x \in \partial \Omega, \,t>0,\\
  &u_{\varepsilon}(x, 0)=u_{0}(x), \quad v_{\varepsilon}(x, 0)=v_{0}(x), & x \in \Omega.
\end{split}
\right.
\end{equation}
Some basic properties of this approximate system are collected in the following lemma.

\begin{lemma}\label{le-glex}
Let  $\chi>0$,  $\mu>0$, $r\in\R$, $n\geq2$ and let \eqref{eq-iva}-\eqref{eq-iva0} hold.
For each $\varepsilon\in(0,1)$, there exists  a unique  $(u_\varepsilon,v_\varepsilon)$ of positive functions with the properties that for any $T>0$ and $p>n$
 \begin{equation*}
\left\{
\begin{split}
&u_\varepsilon\in \mathcal{C}^0\big(\overline{\Omega}\times[0,T]\big)\cap \mathcal{C}^{2,1}\big(\overline{\Omega}\times(0,T]\big),\\
&v_\varepsilon\in \mathcal{C}^0\big([0,T]; W^{1,p}(\overline{\Omega})\big)\cap \mathcal{C}^{2,1}\big(\overline{\Omega}\times(0,T]\big),
\end{split}
\right.
  \end{equation*}
such that $(u_\varepsilon,v_\varepsilon)$ solves the  approximate problem \eqref{eq-app} classically on $\Omega\times(0,\infty)$, and fulfills that
\begin{align}\label{eq-0ul11}
\|u_\varepsilon(\cdot,t)\|_{L^1}+\|v_\varepsilon(\cdot,t)\|_{L^1}\leq C, \quad t>0,\\
\int_0^t\int_\Omega u_\varepsilon v_\varepsilon +\int_0^t\int_\Omega u_\varepsilon^2 \leq C(1+t), \quad t>0\label{eq-0vL-1}
\end{align}
for some $C>0$ independent of $\varepsilon$.
Moreover, it holds that
\begin{align}\label{eq-0velbdd}
v_\varepsilon(\cdot, t)\geq C_T,\quad  \forall t\in(0,T),
\end{align}
where
\begin{align}\label{eq-0bddca}
\widetilde{\eta}:=\inf_{T>0}C_T>0,\quad  \mathrm{if}\,\,\, \eqref{eq-iva1}\,\,\,\mathrm{is\,\, valid}.
\end{align}
\end{lemma}

\begin{proof}
Due to the saturation effect in the second equation,  the contraction mapping principle and  the well-known  properties of the Neumann heat
semigroup (\cite[Lemma 2.1]{LIXIE2022})  enable one to quickly infer  that  the approximate  problem \eqref{eq-app} admits a global  and positive classical solution.

Similar to Lemma \ref{le-uL1},  we see  that
 \eqref{eq-0ul11},
\eqref{eq-0vL-1} and \eqref{eq-0velbdd} hold as well as \eqref{eq-0bddca}   if \eqref{eq-iva1} is true.
\end{proof}

Next, we focus on  some  uniform-in-$\varepsilon$ estimates on the  solution $(u_\varepsilon,v_\varepsilon)$.
\begin{lemma}\label{le-ulnvg}
For any $T>0$, there exists an $\varepsilon$-independent $ C_T>0$  such that
 \begin{align}\label{eq-ulnvg}
\int_0^T\|\ln v_\varepsilon(\cdot,s)\|_{H^1}^2ds+
\int_0^T\| v_\varepsilon^{-1}(\cdot,s)\|_{H^1}^2ds\leq C_T,\\
\label{eq-ulnug}
\int_0^T\|\ln(u_\varepsilon+1)(\cdot,s)\|_{H^1}^2ds\leq C_T,
\end{align}
and that
 for $q>n$,
\begin{align}
\int_{0}^T\left\|\partial_s \ln  v_{\varepsilon}(\cdot, s)\right\|_{\left(W^{1,q}\right)^{\star}} d s+\int_{0}^T\left\|\partial_s v_\varepsilon^{-1}(\cdot,s)\right\|_{\left(W^{1,q}\right)^{\star}} d s \leq C_T\label{eq-tve-1},\\
\int_{0}^T\left\|\partial_s \ln \left(u_{\varepsilon}(\cdot, s)+1\right)\right\|_{\left(W^{1,q}\right)^{\star}} d s \leq C_T.\label{eq-tue}
\end{align}
\end{lemma}

\begin{proof}
The estimates \eqref{eq-ulnvg} and \eqref{eq-tve-1} follow from \cite[Lemma 3.1]{LIXIE2022} directly. To get \eqref{eq-ulnug},
we multiply  the first equation in \eqref{eq-app} by $\frac{1}{1+u_\varepsilon}$ and integrate by parts to obtain
\begin{align*}
\frac{d}{dt}\int_\Omega \ln(1+u_\varepsilon)&+\int_\Omega\frac{u_\varepsilon v_\varepsilon}{1+u_\varepsilon}+\mu \int_\Omega\frac{u_\varepsilon^2}{1+u_\varepsilon}\\
=&\int_{\Omega} \frac{\left|\nabla u_{\varepsilon}\right|^{2}}{\left(u_{\varepsilon}+1\right)^{2}}-\chi\int_{\Omega} \frac{u_{\varepsilon}\nabla u_{\varepsilon}\cdot\nabla\ln v_{\varepsilon} }{1+u_\varepsilon}+r\int_\Omega\frac{u_\varepsilon }{1+u_\varepsilon}+\int_\Omega\frac{h_1}{1+u_\varepsilon},
\end{align*}
which, together with Young's inequality and the facts that  $u_{\varepsilon}, v_{\varepsilon}, h_1\geq0$, warrants
\begin{align*}
&\frac12\int_{\Omega} \frac{\left|\nabla u_{\varepsilon}\right|^{2}}{\left(u_{\varepsilon}+1\right)^{2}}\\
\leq&\frac{d}{dt}\int_\Omega \ln(1+u_\varepsilon)+\frac{\chi^2}{2}\int_{\Omega} \frac{u_{\varepsilon}^{2}}{\left(u_{\varepsilon}+1\right)^{2}}\left|\nabla\ln v_{\varepsilon}\right|^{2}+\int_\Omega\frac{u_\varepsilon v_\varepsilon}{1+u_\varepsilon}+\mu \int_\Omega\frac{u_\varepsilon^2}{1+u_\varepsilon}-r\int_\Omega\frac{u_\varepsilon}{1+u_\varepsilon}\\
\leq&\frac{d}{dt}\int_\Omega \ln(1+u_\varepsilon)+\frac{\chi^2}{2}\int_{\Omega}\left|\nabla\ln v_{\varepsilon}\right|^{2}+ \int_\Omega v_\varepsilon+\mu \int_\Omega u_\varepsilon+|r||\Omega|.
\end{align*}
Notice the facts  $\ln (1+u_0)\geq 0$, $\ln (1+s)\leq s$  and $\ln^2(1+s)\leq 2s$ for all $s\geq 0$; by  the $L^1$-boundedness of $u_\varepsilon$ and $v_\varepsilon$ in \eqref{eq-0ul11}, we integrate the above inequality from $0$ to $t$ to see that
\begin{align*}
&\int_0^t\int_{\Omega}  |\ln (1+u_\varepsilon)|^2+\int_0^t\int_{\Omega} |\nabla \ln (1+u_\varepsilon)|^2
\le\chi^2\int_0^t\int_{\Omega}\left|\nabla\ln v_{\varepsilon}\right|^{2}+C(1+t).
\end{align*}
which along with \eqref{eq-ulnvg} entails  \eqref{eq-ulnug}.

Next, multiplying the first equation in \eqref{eq-app} by $\frac{\varphi}{1+u_\varepsilon}$ with $\varphi \in \mathcal{C}^\infty(\overline{\Omega})$ and integrating  by parts, using
  H\"{o}lder's inequality  and Young's inequality,   we arrive at
 \begin{align*}
\left|\int_\Omega\varphi \partial_t\ln(1+u_\varepsilon)\right|\leq &\|\varphi\|_{L^\infty}\|\nabla\ln(1+u_\varepsilon)\|_{L^2}^2+\chi\|\nabla\ln(1+u_\varepsilon)\|_{L^2}\|\nabla\ln v_\varepsilon\|_{L^2}\|\varphi\|_{L^\infty}
\\
&+\|\nabla\ln(1+u_\varepsilon)\|_{L^2}\|\nabla\varphi\|_{L^2}
+\chi\|\nabla\ln v_\varepsilon\|_{L^2}\|\nabla\varphi\|_{L^2}\\
&+ \|v_\varepsilon\|_{L^1}\|\varphi\|_{L^\infty}+|r|\|\varphi\|_{L^1}+\mu\|u_\varepsilon\|_{L^1}\|\varphi\|_{L^\infty}+\|h_1\|_{L^\infty}\|\varphi\|_{L^1}.
\end{align*}
Then for $q>n$, H\"{o}lder's inequality  and  Sobolev's inequality show that
 \begin{align*}
&\left|\int_\Omega\varphi \partial_t\ln(1+u_\varepsilon)\right|\\ \leq&C\left(\|\nabla\ln(1+u_\varepsilon)\|_{L^2}^2+\|\nabla\ln v_\varepsilon\|_{L^2}^2+\|v_\varepsilon\|_{L^1}+\|u_\varepsilon\|_{L^1}+\|h_1\|_{L^\infty}+1\right)\|\varphi\|_{W^{1,q}}.
\end{align*}
After an integration in time, it follows from  \eqref{eq-iva0}, \eqref{eq-0vL-1}, \eqref{eq-ulnvg} and \eqref{eq-ulnug} that  \eqref{eq-tue} holds.
\end{proof}

Based on the uniform-in-$\varepsilon$ estimates established above, a straightforward compactness argument enables us to extract convergent subsequences in the following
sense.

\begin{lemma}\label{le-coL1}
There exist functions $u\geq0$ and $v>0$ defined on
$\Omega\times(0,T)$ for any $T>0$ and a sequence
$\{\varepsilon_j\}_{j=1}^\infty\subset(0,1)$ such that $\varepsilon_j\rightarrow0$ as $j\rightarrow\infty$, with the properties that for any $T>0$, as $\varepsilon=\varepsilon_j\rightarrow0$,
\begin{align}
&\ln v_\varepsilon\rightarrow \ln v \quad  \mathrm{in}\quad L^2\big(\Omega\times(0,T)\big),\label{eq-dplim31}\\
&\ln v_\varepsilon\rightharpoonup\ln v \quad  \mathrm{in}\quad L^2\big(0,T; H^1(\Omega)\big),\label{eq-dplim5}\\
&v_\varepsilon\rightarrow v \quad a.e.\,\,\,\mathrm{in}\quad \Omega\times(0,T),\label{eq-dplim2v}\\
&v_\varepsilon^{-1}\rightarrow v^{-1} \quad \mathrm{in}\quad L^2\big(\Omega\times(0,T)\big),\label{eq-dplim30}\\
&\ln(1+u_\varepsilon)\rightarrow \ln(1+u) \quad  \mathrm{in}\quad L^2\big(\Omega\times(0,T)\big),\label{eq-dplim}\\
&\ln(1+u_\varepsilon)\rightharpoonup \ln(1+u) \quad  \mathrm{in}\quad L^2\big(0,T; H^1(\Omega)\big),\label{eq-dplim1}\\
&u_\varepsilon\rightarrow u \quad a.e.\,\,\,\mathrm{in}\quad \Omega\times(0,T),\label{eq-dplim2}\\
&u_\varepsilon \rightharpoonup u \quad  \mathrm{in}\quad L^2\big(\Omega\times(0,T)\big),\label{eq-dpL2}\\
&u_\varepsilon\rightarrow u \quad  \mathrm{in}\quad L^\gamma\big(\Omega\times(0,T)\big),\quad \gamma\in[1,2).\label{eq-dplim0}
\end{align}
\end{lemma}

\begin{proof}
Based on the well-known  Aubin-Lions  compactness theorem (\cite{Simon1986}), a combination of \eqref{eq-0velbdd},
 \eqref{eq-ulnvg} and \eqref{eq-tve-1} implies
there exist a subsequence of
$\{\varepsilon_j\}_{j=1}^\infty$ (still expressed as $\{\varepsilon_j\}_{j=1}^\infty$) and a positive function $v$ satisfying $\ln v\in L^2(0,T; H^1(\Omega))$ such  that  \eqref{eq-dplim31}, \eqref{eq-dplim5}, \eqref{eq-dplim2v}  and \eqref{eq-dplim30} hold.

Similarly, thanks to \eqref{eq-ulnug} and \eqref{eq-tue}, the convergence \eqref{eq-dplim}, \eqref{eq-dplim1} and \eqref{eq-dplim2} hold. Moreover, the uniform $L^2$-bound \eqref{eq-0vL-1}, together with \eqref{eq-dplim2},
 ensures the validity of \eqref{eq-dpL2}.
The verification of \eqref{eq-dplim0} can thereupon be completed by  \eqref{eq-0vL-1} and  \eqref{eq-dplim2} via the Vitali convergence theorem (roughly speaking, a.e. convergence along with uniform integrability implies $L^1$-convergence).
\end{proof}

For our purpose, we need to verify that $(u,v)$ satisfies Definition \ref{def-gs}, with particular attention to the energy inequality \eqref{eq-wdef2}. To this end, we shall derive the following  convergence.

\begin{lemma}\label{le-couvL1}
 For  any $T>0$ there exists  a sequence
$\{\varepsilon_j\}_{j=1}^\infty\subset(0,1)$ such that $\varepsilon_j\rightarrow0$ as $j\rightarrow\infty$, with the properties that as $\varepsilon=\varepsilon_j\rightarrow0$,
\begin{align}
&\frac{u_\varepsilon}{1+u_\varepsilon}\rightarrow \frac{u}{1+u} \quad  \mathrm{in}\quad L^p\big(\Omega\times(0,T)\big),\quad 1\leq p<\infty,\label{eq-dpL12}\\
&\frac{\sqrt{1+2u_\varepsilon}}{1+u_\varepsilon}\rightarrow \frac{\sqrt{1+2u}}{1+u} \quad  \mathrm{in}\quad L^p\big(\Omega\times(0,T)\big),\quad 1\leq p<\infty,\label{eq-dpL12+}\\
&\frac{u_\varepsilon^2}{1+u_\varepsilon}\rightarrow \frac{u^2}{1+u} \quad  \mathrm{in}\quad L^1\big(\Omega\times(0,T)\big),\label{eq-dpL120}\\
&\frac{u_\varepsilon v_\varepsilon}{1+u_\varepsilon}\rightarrow \frac{uv}{1+u} \quad  \mathrm{in}\quad L^1\big(\Omega\times(0,T)\big).\label{eq-dpL11}
\end{align}
\end{lemma}

\begin{proof}
For $p\geq 1$,  since $u, u_\varepsilon\geq 0$, we estimate
\begin{align*}
   \left|\frac{u_\varepsilon}{1+u_\varepsilon}-\frac{u}{1+u}\right|^p\leq \left|\frac{u_\varepsilon}{1+u_\varepsilon}-\frac{u}{1+u}\right|\leq \left|u_\varepsilon-u\right|,
\end{align*}
 which, in conjunction with the $L^1$-convergence  \eqref{eq-dplim0}, directly gives
\eqref{eq-dpL12}.
Similarly, we have
\begin{align*}
   \left|\frac{\sqrt{1+2u_\varepsilon}}{1+u_\varepsilon}-\frac{\sqrt{1+2u}}{1+u}\right|^p\leq \left|\frac{\sqrt{1+2u_\varepsilon}}{1+u_\varepsilon}-\frac{\sqrt{1+2u}}{1+u}\right|\leq \left|u_\varepsilon-u\right|,
\end{align*}
which together with  the $L^1$-convergence  \eqref{eq-dplim0} yields
\eqref{eq-dpL12+}, as wished.

For the convergence \eqref{eq-dpL120}, just note
\begin{align*}
   \left|\frac{u_\varepsilon^2}{1+u_\varepsilon}-\frac{u^2}{1+u}\right|=\frac{u+u_\varepsilon+uu_\varepsilon}{1+u+u_\varepsilon+uu_\varepsilon}\left|u_\varepsilon-u\right|\leq \left|u_\varepsilon-u\right|.
\end{align*}
To obtain \eqref{eq-dpL11},  for $q\in(1-\frac2n,1)$, we compute from the $v_\varepsilon$-equation that
\begin{align}\label{eq-Lqve}
\frac1q\frac{d}{dt}\int_\Omega v_\varepsilon^q+\int_\Omega v_\varepsilon^{q}
\geq(1-q)\int_\Omega v_\varepsilon^{q-2}|\nabla v_\varepsilon|^2.
\end{align}
Upon integration and the fact $\|v_\varepsilon(\cdot,t)\|_{L^q}^q\leq\|v_\varepsilon(\cdot,t)\|_{L^1}^q|\Omega|^{1-q}\leq C$ due to \eqref{eq-0ul11},  we see  that
\begin{align*}
\int_0^t\int_\Omega |\nabla v_\varepsilon^\frac q2|^2=\frac{q^2}{4}\int_0^t\int_\Omega v_\varepsilon^{q-2}|\nabla v_\varepsilon|^2\leq \frac{q}{4(1-q)}\int_\Omega v_\varepsilon^q+\frac{q^2}{4(1-q)}\int_0^t\int_\Omega v_\varepsilon^{q}\leq C(1+t).
\end{align*}
Then, by the  $L^1$ bound in  \eqref{eq-0ul11}, we employ  the  Gagliardo-Nirenberg inequality to deduce that
\begin{align*}
\int_0^t\int_\Omega v_\varepsilon^{q+\frac2n}=\int_0^t\left\|v^\frac q2_\varepsilon\right\|_{L^{2+\frac{4}{nq}}}^{2+\frac{4}{nq}}\leq C\int_0^t\left(\left\|v^\frac q2_\varepsilon\right\|_{L^\frac2q}^{\frac{4}{nq}}\left\|\nabla v^\frac q2_\varepsilon\right\|_{L^2}^2+\left\|v^\frac q2_\varepsilon\right\|_{L^\frac2q}^{2+\frac{4}{nq}}\right)\leq C(1+t).
\end{align*}
Observing $q+\frac2n>1$, for any $T>0$, combining the a.e. convergence \eqref{eq-dplim2v} and
the Vitali convergence theorem,  we infer, up to  a subsequence $\varepsilon=\varepsilon_j\rightarrow0$,
\begin{align}\label{veps-v-l1}
v_\varepsilon \rightarrow v \quad  \mathrm{in}\quad L^1\big(\Omega\times(0,T)\big).
\end{align}
Therefore, by  H\"{o}lder's inequality, we conclude
\begin{align*}
&\int_0^t\int_\Omega \left|\frac{u_\varepsilon v_\varepsilon}{1+u_\varepsilon}-\frac{uv}{1+u}\right|\\
&\leq \left(\int_0^t\int_\Omega v_\varepsilon^{q+\frac2n}\right)^\frac{n}{qn+2}\left(\int_0^t\int_\Omega \left|\frac{u_\varepsilon}{1+u_\varepsilon}-\frac{u}{1+u}\right|^\frac{qn+2}{qn+2-n}\right)^\frac{qn+2-n}{qn+2}+\int_0^t\int_\Omega \left|v_\varepsilon-v\right|\\
&\leq C\left(1+t\right)^\frac{n}{qn+2}\left(\int_0^t\int_\Omega \left|\frac{u_\varepsilon}{1+u_\varepsilon}-\frac{u}{1+u}\right|^\frac{qn+2}{qn+2-n}\right)^\frac{qn+2-n}{qn+2}+\int_0^t\int_\Omega \left|v_\varepsilon-v\right|,
\end{align*}
which, together with \eqref{eq-dpL12} and \eqref{veps-v-l1}, guarantees \eqref{eq-dpL11}.
\end{proof}

We are now prepared to suitably pass to the limit in  the approximate problem  \eqref{eq-app}, to achieve the  generalized
solvability.

\begin{lemma}\label{le-vw}
Let $u$ and $v$ be given in Lemma \ref{le-coL1}. For any $T>0$, the IBVP \eqref{eq-r01}  admits at least one global generalized solution $(u,v)$
in the sense of Definitions \ref{def-gs}.
\end{lemma}

\begin{proof}
An integration of the
first equation in \eqref{eq-app} over $\Omega$ shows
\begin{align}\label{eq-xlj}
\frac{d}{dt}\int_\Omega u_\varepsilon + \int_\Omega u_\varepsilon v_\varepsilon+\mu  \int_\Omega u_\varepsilon^2=r \int_\Omega u_\varepsilon+\int_\Omega h_1,
\end{align}
which, upon being integrated over  time,  ensures that
\begin{align*}
\int_\Omega u_\varepsilon +\int_0^t \int_\Omega u_\varepsilon v_\varepsilon +\mu\int_0^t \int_\Omega u_\varepsilon^2=\int_\Omega u_0+r\int_0^t \int_\Omega u_\varepsilon +\int_0^t\int_\Omega h_1.
\end{align*}
Invoking \eqref{eq-0vL-1}, \eqref{eq-dplim2v}, \eqref{eq-dplim2}-\eqref{eq-dplim0}, the lower semi-continuity of norm and Fatou's lemma, up to a subsequence $\varepsilon=\varepsilon_j\to 0$ as $j\to \infty$, we get \eqref{eq-wdef1-1} immediately.

Similarly,  we infer from the second equation in \eqref{eq-app} that for any $t>0$,
\begin{align}\label{eq-xlj0}
\frac{d}{dt}\int_\Omega v_\varepsilon+ \int_\Omega v_\varepsilon = \int_\Omega \frac{u_\varepsilon v_\varepsilon}{1+\varepsilon u_\varepsilon v_\varepsilon} + \int_\Omega h_2.
\end{align}
Adding \eqref{eq-xlj} and \eqref{eq-xlj0} and integrating over time,
 it holds that
\begin{align*}
&\int_\Omega (v_\varepsilon+u_\varepsilon)+
\int_0^t\int_\Omega  v_\varepsilon+\int_0^t \int_\Omega  \frac{\varepsilon u_\varepsilon^2 v_\varepsilon^2}{1+\varepsilon u_\varepsilon v_\varepsilon} +\mu\int_0^t \int_\Omega u_\varepsilon^2 \\
=&\int_\Omega   (u_0 +v_0)+r\int_0^t \int_\Omega u_\varepsilon +\int_0^t\int_\Omega (h_1+h_2),
\end{align*}
which, along with  \eqref{eq-0vL-1}, \eqref{eq-dplim2v}, \eqref{eq-dplim2}-\eqref{eq-dplim0}, the lower semi-continuity of norm  and Fatou's lemma, ensures that \eqref{eq-wdef1-2} holds,  as desired.

To show  \eqref{eq-wdef2}, for $0\leq\varphi\in C_0^\infty(\overline{\Omega}\times [0,T))$ we   test the first equation in \eqref{eq-app} by $\frac{\varphi}{1+u_\varepsilon}$   to get
\begin{align*}
&\frac{d}{dt}\int_\Omega\varphi \ln(1+u_\varepsilon)-\int_\Omega\varphi_t \ln(1+u_\varepsilon)\\
=&-\int_\Omega\frac{\nabla\varphi\cdot\nabla u_\varepsilon}{1+u_\varepsilon}+\int_\Omega\frac{\varphi|\nabla u_\varepsilon|^2}{(1+u_\varepsilon)^2}
+\chi\int_\Omega\frac{u_\varepsilon\nabla\varphi\cdot\nabla\ln v_\varepsilon}{1+u_\varepsilon}-\chi\int_\Omega\frac{\varphi u_\varepsilon\nabla u_\varepsilon\cdot\nabla\ln v_\varepsilon}{(1+u_\varepsilon)^2}\\
&+\int_\Omega\frac{\varphi}{1+u_\varepsilon}\left(- u_\varepsilon v_\varepsilon+ru_\varepsilon-\mu u_\varepsilon^2+h_1\right).
\end{align*}
Similarly, testing the second equation in \eqref{eq-app} by $\frac{\varphi}{v_\varepsilon}$ with any $0\leq\varphi\in C_0^\infty(\overline{\Omega}\times [0,T))$, we have
\begin{align*}
&\frac{d}{dt}\int_\Omega\varphi \ln v_\varepsilon -\int_\Omega\varphi_t \ln v_\varepsilon
=-\int_\Omega\frac{\nabla\varphi\cdot\nabla v_\varepsilon}{v_\varepsilon}+\int_\Omega\frac{\varphi|\nabla v_\varepsilon|^2}{v_\varepsilon^2}
+\int_\Omega\frac{\varphi}{v_\varepsilon}\left(-v_\varepsilon+\frac{u_\varepsilon v_\varepsilon}{1+\varepsilon u_\varepsilon v_\varepsilon}+h_2\right).
\end{align*}
A linear combination  of these two identities, it follows that
\begin{align*}
&\frac{d}{dt}\int_\Omega\varphi \ln(1+u_\varepsilon)-\int_\Omega\varphi_t \ln(1+u_\varepsilon)+\frac{\chi^2}{4}\left(\frac{d}{dt}\int_\Omega\varphi \ln v_\varepsilon -\int_\Omega\varphi_t \ln v_\varepsilon \right)\\
=&\int_\Omega\frac{\varphi|\nabla u_\varepsilon|^2}{(1+u_\varepsilon)^2}-\chi\int_\Omega\frac{\varphi u_\varepsilon\nabla u_\varepsilon\cdot\nabla\ln v_\varepsilon}{(1+u_\varepsilon)^2}
+\frac{\chi^2}{4}\int_\Omega\frac{\varphi|\nabla v_\varepsilon|^2}{v_\varepsilon^2}\\
&-\int_\Omega\frac{\nabla\varphi\cdot\nabla u_\varepsilon}{1+u_\varepsilon}+\chi\int_\Omega\frac{u_\varepsilon\nabla\varphi\cdot\nabla\ln v_\varepsilon}{1+u_\varepsilon}+\int_\Omega\frac{\varphi}{1+u_\varepsilon}\left(- u_\varepsilon v_\varepsilon+ru_\varepsilon-\mu u_\varepsilon^2+h_1\right)\\
&+\frac{\chi^2}{4}\left\{-\int_\Omega\frac{\nabla\varphi\cdot\nabla v_\varepsilon}{v_\varepsilon}
+\int_\Omega\frac{\varphi}{v_\varepsilon}\left(-v_\varepsilon+\frac{u_\varepsilon v_\varepsilon}{1+\varepsilon u_\varepsilon v_\varepsilon}+h_2\right)\right\},
\end{align*}
which, by integrating  in time and noticing $\varphi|_{t=T}\equiv 0$,  implies that
\begin{align*}
&\quad-\int_\Omega\left(\varphi|_{t=0} \ln(1+u_0)+\frac{\chi^2}{4} \varphi|_{t=0} \ln v_0\right)-\int_0^T\int_\Omega\left(\varphi_t \ln(1+u_\varepsilon)+\frac{\chi^2}{4}\varphi_t \ln v_\varepsilon \right)\\
&=\int_0^T\int_\Omega\left(\frac{\varphi|\nabla u_\varepsilon|^2}{(1+u_\varepsilon)^2}
-\chi\int_\Omega\frac{\varphi u_\varepsilon\nabla u_\varepsilon\cdot\nabla\ln v_\varepsilon}{(1+u_\varepsilon)^2}+\frac{\chi^2}{4}\frac{\varphi|\nabla v_\varepsilon|^2}{v_\varepsilon^2}\right)\\
&\quad+\int_0^T\int_\Omega\left(-\frac{\nabla\varphi\cdot\nabla u_\varepsilon}{1+u_\varepsilon}+\chi\frac{u_\varepsilon\nabla\varphi\cdot\nabla\ln v_\varepsilon}{1+u_\varepsilon}+\frac{\varphi}{1+u_\varepsilon}\left(- u_\varepsilon v_\varepsilon+ru_\varepsilon-\mu u_\varepsilon^2+h_1\right)\right)\\
&\quad+\frac{\chi^2}{4}\int_0^T\int_\Omega\left(-\frac{\nabla\varphi\cdot\nabla v_\varepsilon}{v_\varepsilon}
+\frac{\varphi}{v_\varepsilon}\left(-v_\varepsilon+\frac{u_\varepsilon v_\varepsilon}{1+\varepsilon u_\varepsilon v_\varepsilon}+h_2\right)\right)\\
&=:P_1+P_2+P_3.
\end{align*}
By the uniform bounds in \eqref{eq-ulnvg} and \eqref{eq-ulnug},  upon rearrangements, we find
\begin{align*}
P_1=\int_0^T\int_\Omega\left|\sqrt{\varphi}\nabla \ln (1+u_\varepsilon)-\frac{\chi u_\varepsilon\sqrt{\varphi}}{2(1+u_\varepsilon)}\nabla \ln v_\varepsilon\right|^2+\int_0^T\int_\Omega\left|\frac{\chi\sqrt{\varphi(1+2u_\varepsilon)}}{2(1+u_\varepsilon)}\nabla \ln v_\varepsilon \right|^2\leq C_T.
\end{align*}
Thus, by reflexivity, up to a subsequence of $\varepsilon=\varepsilon_j\rightarrow0$, one has, for some $f, g$, that
\begin{align}\label{gsex-con1}
    \sqrt{\varphi}\nabla \ln (1+u_\varepsilon)-\frac{\chi u_\varepsilon\sqrt{\varphi}}{2(1+u_\varepsilon)}\nabla \ln v_\varepsilon  \rightharpoonup  f \text{ in  } \left(L^2(\Omega\times (0, T))\right)^n,
\end{align}
and
\begin{align}\label{gsex-con2}
\frac{\chi\sqrt{\varphi(1+2u_\varepsilon)}}{2(1+u_\varepsilon)}\nabla \ln v_\varepsilon  \rightharpoonup  g  \text{ in  } \left(L^2(\Omega\times (0, T))\right)^n.
\end{align}
On the other hand, by the convergence in  \eqref{eq-dplim5}, \eqref{eq-dplim1}  and  \eqref{eq-dpL12}, one can readily infer that
\begin{align*}
    \sqrt{\varphi}\nabla \ln (1+u_\varepsilon)-\frac{\chi u_\varepsilon\sqrt{\varphi}}{2(1+u_\varepsilon)}\nabla \ln v_\varepsilon  \rightharpoonup   \sqrt{\varphi}\nabla \ln (1+u)-\frac{\chi u\sqrt{\varphi}}{2(1+u)}\nabla \ln v \text{ in  } \left(L^1(\Omega\times (0, T))\right)^n,
\end{align*}
and,  by the convergence in  \eqref{eq-dplim5}, \eqref{eq-dplim1}  and  \eqref{eq-dpL12+},  that
\begin{align*}
\frac{\chi\sqrt{\varphi(1+2u_\varepsilon)}}{2(1+u_\varepsilon)}\nabla \ln v_\varepsilon  \rightharpoonup  \frac{\chi\sqrt{\varphi(1+2u)}}{2(1+u)}\nabla \ln v \text{ in  } \left(L^1(\Omega\times (0, T))\right)^n.
\end{align*}
Therefore, by testing the indicator  function of any measurable subset of $\Omega\times (0, T)$, we see
\[
f=\sqrt{\varphi}\nabla \ln (1+u)-\frac{\chi u\sqrt{\varphi}}{2(1+u)}\nabla \ln v, \  \   \  g= \frac{\chi\sqrt{\varphi(1+2u)}}{2(1+u)}\nabla \ln v.
\]
This, together with \eqref{gsex-con1} and \eqref{gsex-con2} and the lower  semi-continuity of $L^2$-norms with respect to weak  convergence, entails
\begin{align*}
\liminf _{\varepsilon=\varepsilon_{j} \rightarrow 0} P_1
\ge \int_0^T\int_\Omega\left(\frac{\varphi|\nabla u|^2}{(1+u)^2}
-\chi\int_\Omega\frac{\varphi u\nabla u\cdot\nabla\ln v}{(1+u)^2}+\frac{\chi^2}{4}\frac{\varphi|\nabla v|^2}{v^2}\right).
\end{align*}
Upon a subsequence, as $\varepsilon=\varepsilon_j\rightarrow0$, it follows from the weak convergence \eqref{eq-dplim1} that
\begin{align*}
-\int_0^T\int_\Omega\frac{\nabla\varphi\cdot\nabla u_\varepsilon}{1+u_\varepsilon}\rightarrow-\int_0^T\int_\Omega\frac{\nabla\varphi\cdot\nabla u}{1+u},
\end{align*}
and from the weak convergence \eqref{eq-dplim5} and the strong convergence \eqref{eq-dpL12} that
\[
\chi\int_0^T\int_\Omega \frac{u_\varepsilon\nabla\varphi\cdot\nabla\ln v_\varepsilon}{1+u_\varepsilon}\rightarrow \chi\int_0^T\int_\Omega \frac{u\nabla\varphi\cdot\nabla\ln v}{1+u},
\]
as well as from \eqref{eq-dplim0}, \eqref{eq-dpL12}, \eqref{eq-dpL120} and  \eqref{eq-dpL11}   that
\[
\int_0^T\int_\Omega\frac{\varphi}{1+u_\varepsilon}\left(- u_\varepsilon v_\varepsilon+ru_\varepsilon-\mu u_\varepsilon^2+h_1\right)\rightarrow \int_0^T\int_\Omega\frac{\varphi}{1+u}\left(- u v+ru-\mu u^2+h_1\right).
\]
Similarly,  it holds from \eqref{eq-dplim5}, \eqref{eq-dplim30} and \eqref{eq-dplim0} that
\begin{align*}
\lim _{\varepsilon=\varepsilon_{j} \rightarrow 0} P_3=\frac{\chi^2}{4}\int_0^T\int_\Omega\left(-\frac{\nabla\varphi\cdot\nabla v}{v}
+\frac{\varphi}{v}\left(-v+uv+h_2\right)\right).
\end{align*}
Hence, by collecting these estimates, we arrive at  \eqref{eq-wdef2} in the Definition \ref{def-gs}, as desired.

Finally, we simply proceed along the lines of  \cite[Lemma 3.4]{LIXIE2022}  to establish   the validity of  \eqref{eq-wdef4}.
\end{proof}

\subsection{Eventual smoothness of generalized solutions}

 Here, our goal is  to study the eventual smoothness of
the generalized  solution $(u,v)$   in the two- and three-dimensional settings, under the additional assumptions \eqref{eq-iva1}  and \eqref{eq-iva2}.
We would like to remark that although our generalized solution framework in two-dimensional setting is different from that in \cite{Heihoff2020zamp}, by a small adaptation of  the proof of \cite[Theorem 1.1]{QLI2023ERA}, we can readily show  that such generalized solution is eventually smooth  in two-dimensional setting. Therefore, we here only show the eventual smoothness of generalized solutions in three-dimensional setting by
 beginning with the uniform in $ \varepsilon$ boundedness  of the approximate solutions after a waiting time.

\begin{lemma}\label{le-vLk1}
Let $n=3$ and $b$ be given in \eqref{eq-iva3}. Then for any $k_0\in(1,3)$ and $\varepsilon\in(0,1)$ there exists $C=C(k_0)>0$, independent of $\varepsilon$, such that
\begin{equation}\label{eq-vLk1}
\int_t^{t+1}\|v_\varepsilon(\cdot,s)-b\|_{L^{k_0}}^{\frac {k_0}3}ds\leq C,\quad \forall t>0.
\end{equation}
\end{lemma}

\begin{proof}
For any $q:=\frac{k_0}{3}\in (\frac13,1)$,  integrating  \eqref{eq-Lqve}  from $t$ to $t+1$ and using H\"{o}lder's inequality and the uniform $L^1$-norm of $v_\varepsilon$ in \eqref{eq-0ul11},   we discover
\begin{align*}
 \int_t^{t+1}\int_\Omega |\nabla v_{\varepsilon}^{\frac q2}|^2\leq C,\ \ \  \forall  t>0.
\end{align*}
Then, by the 3D  Poinca\'{e}-Sobolev's inequality
\begin{align*}
\left\|v_{\varepsilon}\right\|_{L^{3q}}^q=\left\|v_{\varepsilon}^{\frac q2}\right\|_{L^6}^2\leq C\left(\left\|\nabla v_{\varepsilon}^{\frac q2}\right\|_{L^2}+\left\|v_{\varepsilon}^{\frac q2}\right\|_{L^\frac2q}\right)^2,
\end{align*}
the uniform $L^1$-bound of $v_\varepsilon$ \eqref{eq-0ul11} implies
\begin{align*}
\int_t^{t+1}\left\|v_{\varepsilon}\right\|_{L^{3q}}^{q}\leq C, \ \  \  \forall t>0.
\end{align*}
Therefore,
\begin{align*}
\int_t^{t+1}\left\|v_\varepsilon-b\right\|_{L^{3q}}^q\leq \int_t^{t+1}\left(\left\|v_\varepsilon\right\|_{L^{3q}}^q+b^q|\Omega|^{\frac13}\right)\leq C, \ \ \    \forall t>0.
\end{align*}
yielding \eqref{eq-vLk1} directly.
\end{proof}

The following  lemma provides some useful information on
the large-time behavior of solutions.

\begin{lemma}\label{le-ul2int}
Let all conditions in Lemma \ref{le-vLk1}, \eqref{eq-iva1}   and \eqref{eq-iva2} be in force, and let $r<\widetilde{\eta}$ with $\widetilde{\eta}$ given in \eqref{eq-0bddca} and $k\in(1,k_0)$ with $k_0$ given in Lemma \ref{le-vLk1}. Then when $t\rightarrow \infty$, it holds that
\begin{equation}\label{eq-xld}
\int_\Omega u_\varepsilon+\int_t^{t+1}\left(\int_\Omega u^2_\varepsilon+ \|v_\varepsilon-b\|_{L^k}^{\frac{k_0}{3}}\right)\rightarrow0
\end{equation}
uniformly in $\varepsilon$.
Consequently, for any $\delta>0$,  there exist $T_\delta=T(\delta)>0$ (large enough, independent of $\varepsilon$) and $t_0\in[T_\delta,T_\delta+1]$ (may depend on $\varepsilon$)  such that
\begin{align}\label{eq-xlds}
\int_\Omega u^2_\varepsilon(\cdot,t_0)+\|v_\varepsilon(\cdot,t_0)-b\|_{L^{k}}^{\frac{k_0}{3}}\leq \delta.
\end{align}
\end{lemma}

\begin{proof}
Notice that $v_\varepsilon\geq \widetilde{\eta}$ by  \eqref{eq-0bddca}, upon integrating the $u_\varepsilon$-equation, we deduce from \eqref{eq-xlj} that
\begin{align}\label{eq-uL2s}
\frac{d}{dt}\int_\Omega u_\varepsilon + (\widetilde{\eta}-r)\int_\Omega u_\varepsilon  +\mu  \int_\Omega u_\varepsilon^2 \leq\int_\Omega h_1,
\end{align}
which, by  $\int_\Omega h\in L^\infty([0,\infty))$ due to \eqref{eq-iva2} and  $\widetilde{\eta}-r>0$, yields  (see \cite[Lemma 2.5]{LX2023M3as} for more details)
\begin{equation}\label{eq-uL2s1}
\int_\Omega u_\varepsilon(\cdot,t)\rightarrow0 \quad\textrm{as}\quad t\rightarrow \infty\quad \mathrm{uniformly\,\, in}\,\, \varepsilon.
\end{equation}
Integrating \eqref{eq-uL2s} once again further shows
\begin{align*}
 \int_\Omega u_\varepsilon(\cdot,t+1)+\mu  \int_t^{t+1}\int_\Omega u_\varepsilon^2 \leq\int_\Omega u_\varepsilon(\cdot,t) +\int_t^{t+1}\int_\Omega h_1,
\end{align*}
which, by means of \eqref{eq-uL2s1}  and \eqref{eq-iva2}, leads to
\begin{align}\label{eq-uL2v}
 \int_t^{t+1}\int_\Omega u_\varepsilon^2\rightarrow0 \quad\textrm{as}\quad t\rightarrow \infty\quad \mathrm{uniformly\,\, in}\,\, \varepsilon.
\end{align}

Then applying the arguments in \cite[Lemma 4.5]{LIXIE2022} to the $v_\varepsilon$-equation, we get
\[
\int_t^{t+1}\|v_\varepsilon-b\|_{L^1}ds\to 0 \quad\textrm{as}\quad t\rightarrow \infty\quad \mathrm{uniformly\,\, in}\,\, \varepsilon.
\]
Now,  for  $k\in(1,k_0)$, by the interpolation inequality
\[
\|v_\varepsilon-b\|_{L^k}^{\frac{k_0}{3}}\leq \|v_\varepsilon-b\|_{L^1}^{\frac{k_0(1-\theta)}{3}}\|v_\varepsilon-b\|_{L^{k_0}}^{\frac{k_0\theta}{3}},\quad \quad \theta=\frac{k_0(k-1)}{k(k_0-1)},
\]
 we use  Lemma \ref{le-vLk1} and H\"{o}lder's inequality as well as the fact that $\frac{k_0}{3}<1$ to see that
  \begin{align*}
\int_{t}^{t+1}\|v_\varepsilon-b\|_{L^k}^{\frac{k_0}{3}}ds\leq& \left(\int_{t}^{t+1}\|v_\varepsilon-b\|_{L^1}^{\frac{k_0}{3}}ds\right)^{1-\theta}\left(\int_{t}^{t+1}\|v_\varepsilon-b\|_{L^{k_0}}^{\frac{k_0}{3}}ds\right)^\theta\\
\leq&C \left(\int_{t}^{t+1}\|v_\varepsilon-b\|_{L^1}^{\frac{k_0}{3}}ds\right)^{1-\theta}\\
\leq & C\left(\int_{t}^{t+1}\|v_\varepsilon-b\|_{L^1}ds\right)^{\frac{k_0(1-\theta)}{3}} \rightarrow0 \quad\textrm{as}\quad t\rightarrow \infty\quad \mathrm{uniformly\,\, in}\,\, \varepsilon.
\end{align*}
This together with \eqref{eq-uL2s1} and \eqref{eq-uL2v} verifies
 \eqref{eq-xld}.
 Consequently,   for any $\delta>0$, there exists $T_\delta>0$ (large enough, independent of $\varepsilon$) such that
  \begin{align*}
\int_{T_\delta}^{T_\delta+1}\int_\Omega u_\varepsilon^2 + \int_{T_\delta}^{T_\delta+1}\|v_\varepsilon(\cdot,s)-b\|_{L^{k}}^{\frac{k_0}{3}}ds\leq \delta,
\end{align*}
 which upon the use of mean value theorem   directly gives rise to  \eqref{eq-xlds}.
\end{proof}

For our lateral purpose,  we shall first  derive a crucial $L^2$-estimate of $u_{\varepsilon}$  as follows.

\begin{lemma}\label{le-gu2L2} Let all conditions in Lemma \ref{le-vLk1}, \eqref{eq-iva1}   and \eqref{eq-iva2} be in force, and let $r<\widetilde{\eta}$ with $\widetilde{\eta}$ given in \eqref{eq-0bddca}.  For $q\in(3,6)$  there exists  $C=C(q)>0$  such that, for any $\varepsilon\in (0,1)$ and $t_1> 0$,
\begin{equation}\label{eq-jgu2L2}
\begin{split}
\int_\Omega u_{\varepsilon}^2(\cdot,t) &\leq e^{-(\widetilde{\eta}-r)(t-t_1)}\int_\Omega u_{\varepsilon}^2(\cdot,t_1)+C\int_{t_1}^t\int_\Omega h_1\\[0.25cm]
&\ \ +C\left(1+\sup_{s\in (t_1, t)} \|\nabla v_\varepsilon(\cdot,s)\|_{L^q}^{\frac{2q}{q-3}}\right)\int_{t_1}^t\int_\Omega u_\varepsilon^2,\quad \quad  \forall t\geq t_1.
\end{split}
\end{equation}
\end{lemma}

\begin{proof}
A direct computations from the $u_\varepsilon$-equation shows that
 \begin{align*}
&\frac{1}{2}\frac{d}{dt}\int_\Omega u^2_{\varepsilon}  + \int_\Omega  |\nabla u_{\varepsilon}|^2+\int_\Omega u^2_{\varepsilon} v_{\varepsilon} +\mu\int_\Omega u^3_{\varepsilon} \\
=&\chi\int_\Omega u_{\varepsilon}\nabla u_{\varepsilon}\cdot v_{\varepsilon}^{-1}\nabla v_{\varepsilon}  +r\int_\Omega u^2_{\varepsilon}+\int_\Omega u_{\varepsilon} h_1,
\end{align*}
which, by fact that $v_{\varepsilon}\geq \widetilde{\eta}$ on $\Omega\times(0, \infty)$, implies, for $q\in(3,6)$, that
 \begin{align*}
&\frac{1}{2}\frac{d}{dt}\int_\Omega u^2_{\varepsilon}  + \int_\Omega  |\nabla u_{\varepsilon}|^2+(\widetilde{\eta}-r)\int_\Omega u^2_{\varepsilon} +\mu\int_\Omega u^3_{\varepsilon} \\
\leq&\chi\int_\Omega u_{\varepsilon}\nabla u_{\varepsilon}\cdot v_{\varepsilon}^{-1}\nabla  v_{\varepsilon}  +\int_\Omega u_{\varepsilon} h_1\\
\leq &\frac{\chi}{\widetilde{\eta}}\|u_{\varepsilon}\|_{L^{\frac{2q}{q-2}}}\|\nabla u_{\varepsilon}\|_{L^2}\|\nabla v_{\varepsilon}\|_{L^{q}}+\|u_{\varepsilon}\|_{L^2}\|h_1\|_{L^2}.
\end{align*}
 Then  Young's inequality  and H\"{o}lder's inequality warrant further
 \begin{align*}
\frac{d}{dt}\int_\Omega u^2_{\varepsilon} + \int_\Omega  |\nabla u_{\varepsilon}|^2+(\widetilde{\eta}-r)\int_\Omega u^2_{\varepsilon}
\leq &\frac{\chi^2}{\eta^2}\|\nabla v_\varepsilon\|_{L^q}^2\|u_{\varepsilon}\|_{L^{\frac{2q}{q-2}}}^2 +\frac{\|h_1\|_{L^2}^2}{\widetilde{\eta}-r},\quad \forall t>0.
\end{align*}
We employ the 3D  Gagliardo-Nirenberg inequality and the Young inequality to estimate
 \begin{align*}
\frac{\chi^2}{\eta^2}\|\nabla v_\varepsilon\|_{L^q}^2\|u_{\varepsilon}\|_{L^{\frac{2q}{q-2}}}^2
&\leq C_1\|\nabla v_\varepsilon\|_{L^q}^2\left(\|u_{\varepsilon}\|_{L^2}^{2(1-\vartheta)}\|\nabla u_{\varepsilon}\|_{L^2}^{2\vartheta}+\|u_{\varepsilon}\|_{L^2}^2\right), \ \ \  \vartheta=\frac{3}{q}\\
&\leq \|\nabla u_{\varepsilon}\|_{L^2}^2+C_2\left(\|\nabla v_\varepsilon\|_{L^q}^2+ \|\nabla v_\varepsilon\|_{L^q}^{\frac{2}{1-\vartheta}}\right)\|u_{\varepsilon}\|_{L^2}^2\\
&\leq \|\nabla u_{\varepsilon}\|_{L^2}^2+C_3\left(1+ \|\nabla v_\varepsilon\|_{L^q}^{\frac{2}{1-\vartheta}}\right)\|u_{\varepsilon}\|_{L^2}^2.
\end{align*}
Since $\|h_1\|_{L^2}^2\leq \|h_1\|_{L^1} \|h_1\|_{L^\infty}$,  collecting these estimates,  with $a=\widetilde{\eta}-r>0$,  we obtain
 \begin{align*}
\frac{d}{dt}\int_\Omega u^2_{\varepsilon}  +a\int_\Omega u^2_{\varepsilon}
\leq b(t):=C_4\left(1+ \|\nabla v_\varepsilon\|_{L^q}^{\frac{2}{1-\vartheta}}\right)\|u_{\varepsilon}\|_{L^2}^2 +C_4\|h_1\|_{L^1},\quad \forall t>0.
\end{align*}
Multiplying this by the factor $e^{at}$ and integrating from $t_1$ to $t$, we use the fact $a>0$ to get
 \begin{align*}
\int_\Omega u^2_{\varepsilon}(\cdot,t)\leq& e^{-a(t-t_1)}\int_\Omega u^2_{\varepsilon}(\cdot,t_1) +e^{-at}\int_{t_1}^te^{as}b(s)ds\\
\leq & e^{-a(t-t_1)}\int_\Omega u^2_{\varepsilon}(\cdot,t_1)+\int_{t_1}^tb(s)ds,\quad \forall t\geq t_1.
\end{align*}
Notice  that
\[
\int_{t_1}^tb(s)ds\leq C_4\left(1+\sup_{s\in (t_1, t)} \|\nabla v_\varepsilon(\cdot,s)\|_{L^q}^{\frac{2}{1-\vartheta}}\right)\int_{t_1}^t\int_\Omega u_\varepsilon^2+C_4\int_{t_1}^t\int_\Omega h_1.
\]
Then substituting back the definitions of $a$ and $\vartheta$, we immediately obtain \eqref{eq-jgu2L2},  as desired.
\end{proof}

Given $\varepsilon\in(0,1)$ and  $\delta\in (0,1]$, for  $t_0$  determined by \eqref{eq-xlds} in Lemma \ref{le-ul2int}, define
\begin{align}\label{eq-gu2L2a}
T_\varepsilon:=T_\varepsilon(t_0, \delta):=\sup\left\{t:\sup_{t_0\leq s\le t}\left\|u_\varepsilon(\cdot,s)\right\|_{L^2}^2\leq 4\delta\right\}.
\end{align}
It follows from continuity  that $T_\varepsilon>t_0$. Our main concern is to show  $T_\varepsilon=\infty$ by contradictory argument.  We will make good use of  \eqref{eq-jgu2L2} in Lemma \ref{le-gu2L2}. To this end, we  next  derive an $L^q$- bound of $\nabla v_\varepsilon(\cdot,t)$ for  $q\in (3,6)$ in the interval $(t_0, T_\varepsilon)$.

\begin{lemma}\label{le-gvL2}
For $\varepsilon\in (0,1)$ and $\delta\in (0,1]$, let $t_0$, determined by \eqref{eq-xlds} in Lemma \ref{le-ul2int}, be large enough,  and let $T_\varepsilon$ be defined by \eqref{eq-gu2L2a}. Given $k\in (1, k_0)$ with $k_0$ given in Lemma \ref{le-vLk1}, for any $q\in(3,6)$  there exists  $C=C(q)>0$  such that
\begin{align}\label{eq-gvL2}
\|\nabla v_\varepsilon(\cdot,t)\|_{L^{q}} \leq C\delta^\frac12\left(1+\left(t-t_0\right)^{-\frac12-\frac32(\frac1k-\frac{1}{q})} e^{-(t-t_0)}\right),\quad  \forall t\in(t_0,T_\varepsilon).
\end{align}
As a result, if $T_\varepsilon<\infty$, then \eqref{eq-gvL2} is also true for $t=T_\epsilon$.
\end{lemma}

\begin{proof}
For $\varepsilon\in (0,1)$, notice from the $v_\varepsilon$-equation in \eqref{eq-app} that
\begin{align*}
(v_\varepsilon-b)_t=\Delta (v_{\varepsilon}-b)-(v_{\varepsilon}-b)+\frac{u_{\varepsilon} v_{\varepsilon}}{1+\varepsilon u_{\varepsilon} v_{\varepsilon}}+(h_2-b).
\end{align*}
 For $k\in (1, k_0)$  and $t\in(t_0, T_{\varepsilon})$,  the properties of the Neumann heat semigroup indicate,  for  $q\in(3,6)$,
\begin{align*}
\|\nabla v_{\varepsilon}(\cdot,t)\|_{L^{q}}
\le&C_1\left(1+\left(t-t_0\right)^{-\frac12-\frac32(\frac1k-\frac{1}{q})}\right)e^{-(t-t_0)}\left\|v_{\varepsilon}(\cdot,t_0)-b\right\|_{L^k}\\
&+C_1\int_{t_0}^t\left(1+(t-s)^{-\frac54+\frac{3}{2q}}\right)e^{-(t-s)}\left\|\frac{u_{\varepsilon} v_{\varepsilon}}{1+\varepsilon u_{\varepsilon} v_{\varepsilon}}+h_2-b\right\|_{L^2}ds.
\end{align*}
Due to the uniform $L^1$-bound of $v_{\varepsilon}$ in  \eqref{eq-0ul11}, the Gagliardo-Nirenberg inequality shows
\begin{align*}
&\left\|\frac{u_{\varepsilon} v_{\varepsilon}}{1+\varepsilon u_{\varepsilon} v_{\varepsilon}}+h_2-b\right\|_{L^2}\\[0.25cm]
&\leq \|u_{\varepsilon}\|_{L^2}\|v_{\varepsilon}\|_{L^\infty}+\|h_2-b\|_{L^2}\\
&\leq C_2\|u_{\varepsilon}\|_{L^2}
\left(\|v_{\varepsilon}\|_{L^1}^{1-\theta}\|\nabla v_{\varepsilon}\|_{L^q}^\theta+\|v_{\varepsilon}\|_{L^1}\right)+\|h_2-b\|_{L^2},  \  \  \theta=\frac{3q}{4q-3}\in (0,1)  \\
&\leq C_3\|u_{\varepsilon}\|_{L^2}\left( \|\nabla v_{\varepsilon}\|_{L^q}^\theta+1\right)+\|h_2-b\|_{L^2}.
\end{align*}
 Recalling  $\delta\in (0,1]$ and $q\in (3,6)$,  the smallness at $t_0$ in  \eqref{eq-xlds} and the fact that $\|u_\varepsilon\|_{L^2}\leq 2\sqrt{\delta}$ on $(t_0, T_\varepsilon)$,  and  the facts that $-\frac54+\frac{3}{2q}>-1$ and $\delta^\frac{3}{k_0}<\delta^\frac12$, we thus conclude that
\begin{align*}
\|\nabla v_{\varepsilon}(\cdot,t)\|_{L^{q}}
\le&C_4\delta^\frac12\left(1+(t-t_0)^{-\frac12-\frac32(\frac1k-\frac{1}{q})}\right)e^{-(t-t_0)}+C_4\delta^\frac12\\
&+C_4\int_{t_0}^t\left(1+(t-s)^{-\frac54+\frac{3}{2q}}\right)e^{-(t-s)}\left(\delta^\frac12\left\|\nabla v_{\varepsilon}\right\|_{L^q}^\theta+\|h_2-b\|_{L^2}\right)ds.
\end{align*}
We now take $\alpha\in(0,\frac{6-q}{2q})$ so that $\left(\frac54-\frac{3}{2q}\right)\times\frac{2}{2-\alpha}<1$.
Then the fact that $\alpha<1$ and  the boundedness of $h_2$ from
  \eqref{eq-iva0}  show  that
\[
\|h_2-b\|_{L^2}=\|h_2-b\|_{L^2}^\alpha\|h_2-b\|_{L^2}^{1-\alpha}\leq C_5\|h_2-b\|_{L^2}^\alpha.
\]
Hence, by invoking H\"{o}lder's inequality,  we have
\begin{align*}
&\int_{t_0}^t\left(1+(t-s)^{-\frac54+\frac{3}{2q}}\right)e^{-(t-s)} \|h_2-b\|_{L^2} ds\\
\leq&C_5\int_{t_0}^t\left(1+(t-s)^{-\frac54+\frac{3}{2q}}\right)e^{-(t-s)} \|h_2-b\|_{L^2}^\alpha ds\\
\leq&C_5\left\{\int_{t_0}^t\left(1+(t-s)^{-\left(\frac54-\frac{3}{2q}\right)\times\frac{2}{2-\alpha}}\right)e^{-\frac{2}{2-\alpha}(t-s)}ds\right\}^{\frac{2-\alpha}{2}}\left\{\int_{t_0}^t \|h_2-b\|_{L^2}^2ds\right\}^{\frac\alpha2}\\
\leq&C_6\left\{\int_{t_0}^t \|h_2-b\|_{L^2}^2ds\right\}^{\frac\alpha2},
\end{align*}
which, along with the fact $\int_0^\infty\int_\Omega |h_2-b|^2<\infty$ in \eqref{eq-iva3} and the fact that $t_0$ is large enough, implies
\begin{align*}
\int_{t_0}^t\left(1+(t-s)^{-\frac54+\frac{3}{2q}}\right)e^{-(t-s)} \|h_2-b\|_{L^2} ds
\leq&C_7\delta^\frac12.
\end{align*}
Collecting these estimates, we arrive at
\begin{align*}
\|\nabla v_{\varepsilon}(\cdot,t)\|_{L^{q}}
\le&C_4\delta^\frac12\left(1+\left(t-t_0\right)^{-\frac12-\frac32(\frac1k-\frac{1}{q})}\right)e^{-(t-t_0)}+C_8\delta^\frac12\\
&+C_4\delta^\frac12\int_{t_0}^t\left(1+(t-s)^{-\frac54+\frac{3}{2q}}\right)e^{-(t-s)}\left\|\nabla v_{\varepsilon}\right\|_{L^q}^\theta ds.
\end{align*}
For any $T\in (t_0, T_\varepsilon)$, setting $K(T):=\sup\left\{\|\nabla v_{\varepsilon}(\cdot, t)\|_{L^{q}}:\  t\in(t_0,T)\right\}$ (it is finite), it follows that
\begin{align*}
\|\nabla v_{\varepsilon}(\cdot,t)\|_{L^{q}}
\le&C_8\delta^\frac12+C_4\delta^\frac12\left(t-t_0\right)^{-\frac12-\frac32(\frac1k-\frac{1}{q})} e^{-(t-t_0)} +C_{9}\delta^\frac12 K^\theta(T),\quad \forall t\in(t_0,T)
\end{align*}
If $C_{9}\delta^\frac12 K^\theta(T) \leq C_8\delta^\frac12+C_4\delta^\frac12\left(t-t_0\right)^{-\frac12-\frac32(\frac1k-\frac{1}{q})} e^{-(t-t_0)}$, then
 \begin{align}\label{eq-gvlk}
\|\nabla v_{\varepsilon}(\cdot,t)\|_{L^{q}}
\le&2C_8\delta^\frac12+2C_4\delta^\frac12\left(t-t_0\right)^{-\frac12-\frac32(\frac1k-\frac{1}{q})} e^{-(t-t_0)},\quad \forall t\in(t_0, T);
\end{align}
otherwise
\begin{align*}
\|\nabla v_{\varepsilon}(\cdot,t)\|_{L^{q}}
\le&2C_{9}\delta^\frac12 K^\theta(T),\quad \forall t\in(t_0,T),
\end{align*}
which, together with the facts that $\delta\in (0,1]$ and $\theta\in (0,1)$, further leads to
\begin{align*}
K(T)
\le&\left(2C_{9}\delta^\frac12 \right)^{\frac{1}{1-\theta}}\leq C_{10}\delta^\frac12.
\end{align*}
This, combining with \eqref{eq-gvlk} and the abitriness of $T\in (t_0, T_\varepsilon)$, indeed ensures \eqref{eq-gvL2}.
\end{proof}

With the key ingredients provided in Lemmas \ref{le-ul2int}, \ref{le-gu2L2}  and \ref{le-gvL2}, we can now show and prove our main concern as below.

\begin{claim}\label{le-claim}
Under Lemmas \ref{le-ul2int} and \ref{le-gvL2}, for every  $\varepsilon\in(0,1)$ and  $\delta\in (0,1]$, there exists  $t_0>0$ satisfying  \eqref{eq-xlds}  such  that $T_\varepsilon=T_\varepsilon(t_0, \delta) =\infty$, where $T_\varepsilon$ is defined by \eqref{eq-gu2L2a}.
\end{claim}
\begin{proof}
Suppose there exist $\varepsilon_0\in (0,1)$ and $\delta_0\in (0,1]$ such that, for any $t_0$ satisfying  \eqref{eq-xlds}, we have $T_{\varepsilon_0}=T_{\varepsilon_0}(\delta_0, t_0)<\infty$. Then we set
$t_1=\frac12(t_0+T_{\varepsilon_0})$, and then apply Lemma \ref{le-ul2int} with $\varepsilon=\varepsilon_0$ and $q=4$ to find some $A>0$ such that
\begin{equation}\label{eq-jgu2L2-te}
\begin{split}
\int_\Omega u_{\varepsilon_0}^2(\cdot,t) &\leq e^{-(\widetilde{\eta}-r)(t-t_1)}\int_\Omega u_{\varepsilon_0}^2(\cdot,t_1)+A\int_{t_1}^t\int_\Omega h_1\\[0.25cm]
&\ \ +A\left(1+\sup_{s\in (t_1, t)} \|\nabla v_{\varepsilon_0}(\cdot,s)\|_{L^4}^8\right)\int_{t_1}^t\int_\Omega u_{\varepsilon_0}^2\\
&\leq 4\delta_0e^{-(\widetilde{\eta}-r)(t-t_1)}+A\int_{t_1}^{T_{\varepsilon_0}}\int_\Omega h_1\\\
&\ \ +A\left(1+\sup_{s\in (t_1, {T_{\varepsilon_0}})} \|\nabla v_{\varepsilon_0}(\cdot,s)\|_{L^4}^8\right)\int_{t_1}^{T_{\varepsilon_0}}\int_\Omega u_{\varepsilon_0}^2,\quad  \forall t\in [t_1, {T_{\varepsilon_0}}].
\end{split}
\end{equation}
Here, we also used the fact $t_1\in (t_0, {T_{\varepsilon_0}})$ to amplify the first term on the right-hand side.

Next, since $\int_0^\infty\int_\Omega h_1<\infty$ due to  \eqref{eq-iva2}, we choose  $t_0>0$ large enough so that
\[
A\int_{t_1}^{T_{\varepsilon_0}}\int_\Omega h_1\leq \delta_0\left(1-e^{-\frac{(\eta-r)}{2}(T_{\varepsilon_0}-t_0)}\right),
\]
and,  apply   Lemma \ref{le-gvL2}  with $(k,q)=(2,4)$ to find some $B>0$ such that
\begin{align*}
\sup_{t\in (t_1, T_{\varepsilon_0})}\|\nabla v_{\varepsilon_0}(\cdot,t)\|_{L^4} \leq B\delta_0^\frac12\left(1+\left(\frac{1}{2}(T_{\varepsilon_0}-t_0)\right)^{-\frac78} \right)<\infty.
\end{align*}
Then, due to \eqref{eq-xld} of Lemma \ref{le-ul2int},  such large $t_0$ (further enlarging if necessary) also  fulfills
\[
A\left(1+\sup_{s\in (t_1, T_{\varepsilon_0})} \|\nabla v_{\varepsilon_0}(\cdot,s)\|_{L^4}^8\right)\int_{t_1}^{T_{\varepsilon_0}}\int_\Omega u_{\varepsilon_0}^2\leq \delta_0\left(1-e^{-\frac{(\eta-r)}{2}(T_{\varepsilon_0}-t_0)}\right).
\]
 Substituting these estimates into \eqref{eq-jgu2L2-te}, we achieve, for sufficiently large $t_0$ satisfying  \eqref{eq-xlds}, that
 \begin{align*}
\int_\Omega u_{\varepsilon_0}^2(\cdot,t)
\leq 4\delta_0e^{-(\widetilde{\eta}-r)(t-t_1)}+2\delta_0\left(1-e^{-\frac{(\eta-r)}{2}(T_{\varepsilon_0}-t_0)}\right),\quad  \forall t\in [t_1, {T_{\varepsilon_0}}].
\end{align*}
This along wih the fact $t_1=\frac12(t_0+T_{\varepsilon_0})$ in particular shows
 \begin{align*}
\int_\Omega u_{\varepsilon_0}^2(\cdot,T_{\varepsilon_0})
\leq 2\delta_0+2\delta_0e^{-\frac{(\eta-r)}{2}(T_{\varepsilon_0}-t_0)}<4\delta_0,
\end{align*}
a clear contradiction to the definition of $T_{\epsilon_0}$ in \eqref{eq-gu2L2a} with $\varepsilon=\varepsilon_0$, and  thus finishing the proof.
\end{proof}

Thanks to Claim \ref{le-claim},  we can
proceed to derive more uniform in $\varepsilon$-estimates for $u_\varepsilon$ and $v_\varepsilon$.

\begin{lemma}\label{le-Vubedp}
 Under Lemmas \ref{le-ul2int} and \ref{le-gvL2},  there exists $\widehat{t}_1>0$,  independent of  $\varepsilon\in (0, 1)$, fulfilling that
whenever $t\geq \widehat{t}_1$, then for any $p\in (1,3)$ and $q\in (1, 6) $ one can find $C=C(p,q)>0$, independent of  $\varepsilon$, such that
\begin{align}\label{eq-gveLp}
\|u_\varepsilon(\cdot,t)\|_{L^p}+\|v_\varepsilon(\cdot,t)\|_{W^{1,q}} \leq C.
\end{align}
\end{lemma}

\begin{proof}
 For $\varepsilon\in (0,1)$ and  $\delta=1$, Claim \ref{le-claim} guarantees there exists $t_0>0$ satisfying  \eqref{eq-xlds}  such  that $T_\varepsilon=T_\varepsilon(t_0, 1) =\infty$. Then for $q\in(3,6)$, it follows from  \eqref{eq-gvL2} that
 \begin{align*}
\|\nabla v_{\varepsilon}(\cdot, t)\|_{L^q}
\leq C_q,\quad t\geq t_0+1.
\end{align*}
Recalling $t_0\in [T_\delta, T_\delta+1]$  with $T_\delta$ independent of $\varepsilon$, this further shows
 \begin{align}\label{eq-gveL4}
\|\nabla v_{\varepsilon}(\cdot, t)\|_{L^q}
\leq &C_q,\quad t\geq T_\delta+2=:t_1.
\end{align}
Utilizing the well-known properties the Neumann heat semigroup, for any $t>t_1$, $p\in[2,3)$ and $q_1\in (\frac{3p}{3+p},\frac32)$, we deduce from the $u_\varepsilon$-equation that
\begin{align*}
\|u_\varepsilon(\cdot,t)\|_{L^p} \leq &C\left(1+(t-t_1)^{-\frac32(\frac12-\frac1p)}\right)e^{-(t-t_1)}\|u_\varepsilon(\cdot,t_1)\|_{L^2}\\
&+C\int_{t_1}^t\left(1+(t-s)^{-\frac12-\frac{3}{2}(\frac{1}{q_1}-\frac1p)}\right)e^{-(t- s)}\left\|\frac{u_\varepsilon}{v_\varepsilon}\nabla  v_\varepsilon\right\|_{L^{q_1}}ds\\
&+C\int_{t_1}^t\left(1+(t-s)^{-\frac{3}{2}(\frac12-\frac1p)}\right)e^{-(t- s)}\|(r+1)u_\varepsilon+h_1\|_{L^2}ds,
\end{align*}
which, by combining with the regularity of $h_1$ in \eqref{eq-iva0} and  the fact $\|u_\varepsilon\|_{L^2}\leq 2$ on $(t_0, \infty)$, entails
\begin{align*}
\|u_\varepsilon(\cdot,t)\|_{L^p} \leq &C\left(1+(t-t_1)^{-\frac32(\frac12-\frac1p)}\right) +C\int_{t_1}^t\left(1+(t-s)^{-\frac12-\frac{3}{2}(\frac{1}{q_1}-\frac1p)}\right)e^{- (t- s)}\left\|\frac{u_\varepsilon}{v_\varepsilon}\nabla  v_\varepsilon\right\|_{L^{q_1}}ds.
\end{align*}
We  observe the choice of $q_1$ ensures  that $3<\frac{2q_1}{2-q_1}<6$, and thereby, by \eqref{eq-gveL4} and the  fact that $v_{\varepsilon}\geq \widetilde{\eta}$ on $\Omega\times(0, \infty)$, we  use  H\"{o}lder's inequality  to deduce that
\[
\left\|\frac{u_\varepsilon}{v_\varepsilon}\nabla  v_\varepsilon\right\|_{L^{q_1}}\leq\left\|u_\varepsilon\right\|_{L^2}\left\|v_\varepsilon^{-1}\right\|_{L^\infty}\left\|\nabla v_\varepsilon\right\|_{L^{\frac{2q_1}{2-q_1}}}\leq C,\quad t\geq t_1,
\]
which, along with the fact $-\frac12-\frac{3}{2}(\frac{1}{q_1}-\frac1p)>-1$, finally enables us to conclude that
\begin{align*}
\|u_\varepsilon(\cdot,t)\|_{L^p} \leq &C_p\quad t\geq  \widehat{t}_1:= t_1+1.
\end{align*}
This, together with \eqref{eq-gveL4}, H\"{o}lder's inequality and
 Poinca\'{e}-Sobolev's inequality
\begin{align*}
\left\|v_{\varepsilon}\right\|_{L^q}\leq C\left(\left\|\nabla v_{\varepsilon}\right\|_{L^q}+\left\|v_\varepsilon\right\|_{L^1}\right),
\end{align*}
gives rise to  \eqref{eq-gveLp}, as sought.
\end{proof}

Based on the  uniform  estimate in  \eqref{eq-gveLp}, we can
derive uniform in $\varepsilon$ estimates for $u_\varepsilon$ with respect to the norm in $L^\infty$, and for $v_\varepsilon$ with respect to the norm in $W^{1,\infty}$.

\begin{lemma}\label{le-Vubed}
  Under Lemmas \ref{le-ul2int} and \ref{le-gvL2},  there exists $\widehat{t}_2>0$,  independent of  $\varepsilon$, fulfilling that
whenever $t\geq\widehat{t}_2$， one can find $C>0$, independent of  $\varepsilon$, such that
\begin{align}\label{eq-gveL2}
\|u_\varepsilon(\cdot,t)\|_{L^\infty}+\|v_\varepsilon(\cdot,t)\|_{W^{1,\infty}} \leq C.
\end{align}
\end{lemma}

\begin{proof} The proof is similar to the proceeding lemma.  Indeed, for $q_2\in(\frac32,2)$, we can choose $k_1\in (1,3)$ and $k_2\in (1,6) $ such that $\frac{1}{q_2}=\frac{1}{k_1}+\frac{1}{k_2}$. Hence,   thanks to  \eqref{eq-gveLp}    and the  fact that $v_{\varepsilon}\geq \widetilde{\eta}$ on $\Omega\times(0, \infty)$, we  see  that
\[
\left\|\frac{u_\varepsilon}{v_\varepsilon}\nabla  v_\varepsilon\right\|_{L^{q_2}}\leq\left\|u_\varepsilon\right\|_{L^{k_1}}\left\|v_\varepsilon^{-1}\right\|_{L^\infty}\left\|\nabla v_\varepsilon\right\|_{L^{k_2}}\leq C,\quad t\geq \widehat t_1.
\]
Based on this, for $3
\leq  p_1<\frac{3q_2}{3-q_2}$, which imlies $-\frac12-\frac{3}{2}(\frac{1}{q_2}-\frac{1}{p_1})>-1$ , as before, we use Neumann semigroup estimates to infer that
\begin{align*}
\|u_\varepsilon(\cdot,t)\|_{L^{p_1}} \leq &C\left(1+(t-\widehat t_1)^{-\frac32(\frac12-\frac{1}{p_1})}\right)e^{-(t-\widehat t_1)}\|u_\varepsilon(\cdot,\widehat t_1)\|_{L^2}\\
&+C\int_{\widehat t_1}^t\left(1+(t-s)^{-\frac12-\frac{3}{2}(\frac{1}{q_2}-\frac{1}{p_1})}\right)e^{-(t- s)}\|\frac{u_\varepsilon}{v_\varepsilon}\nabla  v_\varepsilon\|_{L^{q_2}}ds\\
&+C\int_{\widehat t_1}^t\left(1+(t-s)^{-\frac{3}{2}(\frac12-\frac{1}{p_1})}\right)e^{-(t- s)}\|(r+1)u_\varepsilon+h_1\|_{L^2}ds\\
\leq& C\left(1+(t-\widehat t_1)^{-\frac32(\frac12-\frac{1}{p_1})}\right),\quad t>\widehat t_1,
\end{align*}
which implies
\begin{align*}
\|u_\varepsilon(\cdot,t)\|_{L^{p_1}}
\leq&C,\quad t\geq t_2:=\widehat t_1+1.
\end{align*}
By taking $q_2$ sufficiently closed to $2$,  then $p_1$ can be close to $6$ , and thereby for any $p_1\in (1,6)$,
\begin{align}\label{eq-upeL6}
\|u_\varepsilon(\cdot,t)\|_{L^{p_1}}
\leq&C,\quad t\geq t_2.
\end{align}

Notice that $W^{1,4}(\Omega)\hookrightarrow L^\infty(\Omega)$, we use the derived $(L^4, W^{1,4})$-boundedness of $(u,v)$  to get
\begin{align*}
\left\|\frac{u_\varepsilon v_\varepsilon}{1+\varepsilon u_\varepsilon v_\varepsilon}+h_2\right\|_{L^4}\leq& \|u_\varepsilon\|_{L^4}\|v_\varepsilon\|_{L^\infty}+\|h_2\|_{L^4}\\
\leq& C\|u_\varepsilon\|_{L^4}\|v_\varepsilon\|_{W^{1,4}}+\|h_2\|_{L^4}\leq C,\quad t\geq t_2,
\end{align*}
Hence, by employing the semigroup representation of $v_\varepsilon$, we deduce  that
\begin{align*}
\|\nabla v_\varepsilon(\cdot,t)\|_{L^\infty} \leq &C\left(1+(t-t_2)^{-\frac3 4}\right)\|\nabla v_\varepsilon(\cdot,t_2)\|_{L^4}\\
&+C\int_{t_2}^t\left(1+(t-s)^{-\frac78}\right)e^{-(t-s)}\left\|\frac{u_\varepsilon v_\varepsilon}{1+\varepsilon u_\varepsilon v_\varepsilon}+h_2\right\|_{L^4}ds\\
\leq& C\left(1+(t-t_2)^{-\frac34}\right),\quad \forall t\geq t_2,
\end{align*}
giving directly
\begin{align}\label{eq-vgLinf}
\|\nabla v_\varepsilon(\cdot,t)\|_{L^\infty}
\leq& C,\quad t\geq t_3:=t_2+1.
\end{align}

Finally, based on  the $(L^4, L^\infty)$-boundedness of $(u_\varepsilon, \nabla  v_\varepsilon)$ in \eqref{eq-upeL6} and \eqref{eq-vgLinf},  and the fact that $v_{\varepsilon}\geq \widetilde{\eta}$ on $\Omega\times(0, \infty)$  by \eqref{eq-0bddca},  we see that
\begin{align*}
\|u_\varepsilon(\cdot,t)\|_{L^\infty} \leq &C\left(1+(t-t_3)^{-\frac34}\right)e^{-(t-t_3)}\|u_\varepsilon(\cdot,t_3)\|_{L^2}\\
&+C\int_{t_3}^t\left(1+(t-s)^{-\frac{7}{8}}\right)e^{- (t-s)}\|\frac{u_\varepsilon}{v_\varepsilon}\nabla v_\varepsilon\|_{L^4}ds\\
&+C\int_{t_3}^t\left(1+(t-s)^{-\frac{3}{4}}\right)e^{-(t- s)}\|(r+1)u_\varepsilon+h_1\|_{L^2}ds\\
\leq&C\left(1+(t-t_3)^{-\frac34}\right),\quad \forall t\geq t_3.
\end{align*}
This simply  ensures that
\begin{align*}
\|u_\varepsilon(\cdot,t)\|_{L^\infty}
\leq&C,\quad t\geq \widehat t_2:=t_3+1,
\end{align*}
which, combining with  \eqref{eq-0ul11}, \eqref{eq-vgLinf} and the interpolation inequality, entails  \eqref{eq-gveL2}.
\end{proof}

A straightforward consequence of Lemma \ref{le-Vubed} is that the
generalized solution $(u,v)$ established in Lemma \ref{le-vw} enjoys  the desired regularity
in Theorem \ref{th-ggs}.

\begin{lemma}\label{le-uvcla}
 Under Lemmas \ref{le-ul2int} and \ref{le-gvL2},  there exist $\widehat{t}_3>0$ and $C>0$  (independent of  $\varepsilon$) such that
\begin{align}
\|u\|_{\mathcal{C}^{2,1}(\Omega\times[t,t+1])}+\|v\|_{\mathcal{C}^{2,1}(\Omega\times[t,t+1])}\leq C,\quad t\geq \widehat{t}_3,\label{eq-LBouC}\\
\|u(\cdot,t)\|_{L^\infty}+\|v(\cdot,t)\|_{W^{1,\infty}} \leq C,\quad t\geq \widehat{t}_3.\label{eq-ULinWo}
\end{align}
 Moreover,  $(u,v)$ solves the  boundary value problem in \eqref{eq-r01} classically in $\Omega\times[\widehat{t}_3,\infty)$.
\end{lemma}

\begin{proof}
Note that   $v_{\varepsilon}\geq \widetilde{\eta}$ on $\Omega\times(0, \infty)$, then invoking Lemma \ref{le-Vubed} and
 the parabolic Schauder estimates (e.g., \cite{Lady1968}),  a   standard bootstrap technique ensures that
there exist $C>0$ (independent of $\varepsilon$), $\widehat t_3>0$ and $\alpha\in(0,1)$ such that
\begin{align*}
\|u_\varepsilon\|_{C^{2+\alpha,1+\frac{\alpha}{2}}(\bar\Omega\times[t,t+1])}+\|v_\varepsilon\|_{C^{2+\alpha,1+\frac{\alpha}{2}}(\bar\Omega\times[t,t+1])}\leq C,\quad t\geq \widehat t_3,
\end{align*}
which, in joint with the convergence properties provided by  Lemma \ref{le-coL1} and the Arzel\`{a}-Ascoli theorem, implies that for any $t\geq \widehat t_3$ there exists a subsequence of
$\{\varepsilon_j\}_{j=1}^\infty$ (still expressed as $\{\varepsilon_j\}_{j=1}^\infty$) such that  as $\varepsilon=\varepsilon_j\rightarrow0$,
\begin{align}
u_{\varepsilon} \rightarrow u\quad \mathrm{in}\quad C^{2,1}(\bar\Omega\times[t,t+1]),\label{eq-uc21}\\
v_{\varepsilon} \rightarrow v\quad \mathrm{in}\quad C^{2,1}(\bar\Omega\times[t,t+1]).\label{eq-vc21}
\end{align}
Collecting these and Lemma \ref{le-Vubed},    the desired  properties   \eqref{eq-LBouC} and \eqref{eq-ULinWo} holds immediately. Also, we observe  that $(u,v)$ solves the  boundary value problem in \eqref{eq-r01} classically in $\Omega\times[\widehat t_3,\infty)$.
\end{proof}

Finally, we   verify the long-time behavior of the generalized solution described in Theorem \ref{th-ggs}.

\begin{lemma}\label{le-de2}
Besides the conditions in  Lemmas \ref{le-ul2int} and \ref{le-gvL2}, let   \eqref{eq-iva3} also be in force,   the  eventual smooth  solution  $(u,v)$ satisfies that
\begin{align}\label{eq-gdeu}
\|u(\cdot,t)\|_{L^\infty}+\|v(\cdot,t)-b\|_{W^{1,\infty}}\rightarrow0,\quad  \mathrm{as}\,\,\, t\rightarrow\infty.
\end{align}
\end{lemma}

\begin{proof}
Testing the second equation in the approximate system \eqref{eq-app} by $2(v_\varepsilon-b)$, we obtain
\begin{align*}
&\frac{d}{dt}\int_\Omega(v_\varepsilon-b)^2+2\int_\Omega|\nabla v_\varepsilon|^2+2\int_\Omega(v_\varepsilon-b)^2\\
=&2\int_\Omega\frac{u_\varepsilon v_\varepsilon}{1+\varepsilon u_\varepsilon v_\varepsilon}(v_\varepsilon-b)+2\int_\Omega (h_2-b)(v_\varepsilon-b)\\
\leq&2\|u_\varepsilon\|_{L^2}\|v_\varepsilon\|_{L^\infty}\|v_\varepsilon-b\|_{L^2}+2\|v_\varepsilon-b\|_{L^2}\|h_2-b\|_{L^2},
\end{align*}
which, by employing Young's inequality, gives us
\begin{align*}
\frac{d}{dt}\int_\Omega(v_\varepsilon-b)^2+2\int_\Omega|\nabla v_\varepsilon|^2+\int_\Omega(v_\varepsilon-b)^2\leq&2\|u_\varepsilon\|_{L^2}^2\|v_\varepsilon\|_{L^\infty}^2+2\|h_2-b\|_{L^2}^2.
\end{align*}
By means of \eqref{eq-gveL2}, one can pick $C>0$ (independent of $\varepsilon$) and $t_4>0$ such that
\[
\|v_\varepsilon(\cdot,t)\|_{L^\infty}\leq C, \quad t\geq t_4,
\]
which further shows that
\begin{align*}
\frac{d}{dt}\int_\Omega(v_\varepsilon-b)^2+2\int_\Omega|\nabla v_\varepsilon|^2+\int_\Omega(v_\varepsilon-b)^2\leq&C\int_\Omega u_\varepsilon^2 +2\int_\Omega (h_2-b)^2, \quad t\geq t_4.
\end{align*}
In light of \eqref{eq-xld} and \eqref{eq-iva3}, it holds that
\[
C\int_t^{t+1}\int_\Omega u_\varepsilon^2+2\int_t^{t+1}\int_\Omega (h_2-b)^2\rightarrow0 \quad\textrm{as}\quad t\rightarrow \infty\quad \mathrm{uniformly\,\, in}\,\, \varepsilon.
\]
Hence, invoking \eqref{eq-iva}, \eqref{eq-gveL2} and \cite[Lemma 4.2]{LIXIE2022}, it follows that
\begin{align*}
\int_\Omega|v_\varepsilon(\cdot,t)-b|^2\rightarrow0 \quad\textrm{as}\quad t\rightarrow \infty\quad \mathrm{uniformly\,\, in}\,\, \varepsilon.
\end{align*}
This, together with the convergence in  Lemma \ref{le-coL1} and \eqref{eq-vc21}, entails that
\begin{align*}
\int_\Omega|v(\cdot,t)-b|^2 \rightarrow0 \quad\textrm{as}\quad t\rightarrow \infty,
\end{align*}
which, in conjunction with  the regularity \eqref{eq-LBouC}, yields that
\begin{align}\label{eq-w1qv}
\|v(\cdot,t)-b\|_{W^{1,\infty}}\rightarrow0 \quad\textrm{as}\quad t\rightarrow \infty.
\end{align}

On the other hand, recalling \eqref{eq-xld}, Lemma \ref{le-coL1} and \eqref{eq-uc21}, we conclude that
\begin{align*}
\int_\Omega u(\cdot,t) \rightarrow0 \quad\textrm{as}\quad t\rightarrow \infty,
\end{align*}
which, with the aid of \eqref{eq-LBouC} and \eqref{eq-ULinWo}, implies that
\begin{align*}
\|u(\cdot,t)\|_{L^\infty}\rightarrow0 \quad\textrm{as}\quad t\rightarrow \infty.
\end{align*}
This, together with \eqref{eq-w1qv}, ensures \eqref{eq-gdeu} immediately.
\end{proof}

By combining with the results of Lemma \ref{le-uvcla} and  Lemma \ref{le-de2}, we thus complete the proof of Theorem \ref{th-ggs}.

\begin{proof}[Proof of Theorem \ref{th-ggs}]
The eventual smoothness of generalized solution described in Theorem \ref{th-ggs}  follows  from Lemma \ref{le-uvcla} directly, and the asymptotic behaviors follow from  Lemma \ref{le-de2}.
\end{proof}

\section*{Acknowledgements}
 BL was supported by  Natural Science Foundation of Ningbo Municipality (No. 2025J199). TX was supported by  National Natural Science Foundation of China ( Nos. 12371214 and  12071476).

\section*{Data Availability}
No new data were created or analysed during this study. Data sharing is not applicable to this article.

\section*{Declarations: Conflict of interest}
On behalf of all authors, the corresponding author states that there is no conflict of interest.

\end{document}